\documentclass[11pt]{article}
\usepackage[T1]{fontenc}
\usepackage{algorithmic,algorithm}
\usepackage{amsmath,amsfonts,amsthm,mathrsfs,amssymb,stmaryrd}
\usepackage{cite}
\usepackage{multirow}
\usepackage{graphicx,float}
\usepackage{subfigure}
\usepackage{placeins}
\usepackage{color}
\usepackage{indentfirst}
\usepackage[bookmarks=true]{hyperref}
\usepackage{bm}
\usepackage{lineno}
\numberwithin{equation}{section}
\newcommand{\R}{{\mathbb R}}

\newcommand{\C}{{\mathbb C}}

\newcommand{\be}{\begin{eqnarray}}
\newcommand{\ben}{\begin{eqnarray*}}
\newcommand{\en}{\end{eqnarray}}
\newcommand{\enn}{\end{eqnarray*}}

\newcommand{\pa}{\partial}

\newcommand{\ov}{\overline}

\newtheorem{theorem}{Theorem}[section]
\newtheorem{lemma}[theorem]{Lemma}

\newtheorem{definition}[theorem]{Definition}
\newtheorem{remark}[theorem]{Remark}

\newtheorem{condition}[theorem]{Condition}

\definecolor{r1}{rgb}{0.000,0.000,0.000}

\begin{document}
\renewcommand{\theequation}{\arabic{section}.\arabic{equation}}
\begin{titlepage}
  \title{Multiple-scattering frequency-time hybrid solver for the\\
    wave equation in interior domains}

\author{Oscar P. Bruno\thanks{Department of Computing \& Mathematical Sciences, California Institute of Technology, 1200 East California Blvd., CA 91125, United States. Email:{\tt obruno@caltech.edu}}\;
and Tao Yin\thanks{LSEC, Institute of Computational Mathematics and Scientific/Engineering Computing, Academy of Mathematics and Systems Science, Chinese Academy of Sciences, Beijing 100190, China. Email:{\tt yintao@lsec.cc.ac.cn}}}
\end{titlepage}
\maketitle
%
\begin{abstract}
  This paper proposes a frequency-time hybrid solver for the
  time-dependent wave equation in two-dimensional {\em interior
    spatial domains}. The approach relies on four main elements,
  namely, 1)~A multiple scattering strategy that decomposes a given
  \textcolor{r1}{{\it interior}} time-domain problem into a sequence of {\em limited-duration}
  time-domain problems of scattering by overlapping \textcolor{r1}{open arcs}, each
  one of which is reduced (by means of the Fourier transform) to a
  sequence of {\em Helmholtz frequency-domain problems};
  2)~Boundary integral equations on overlapping boundary patches for
  the solution of the frequency-domain problems in point~1); 3)~A
  smooth {\em ``Time-windowing and recentering''} methodology that
  enables both treatment of incident signals of long duration and long
  time simulation; and, 4)~A Fourier transform algorithm that delivers
  numerically dispersionless, {\em spectrally-accurate time evolution}
  for given incident fields. By recasting the interior time-domain
  problem in terms of a sequence of open-arc multiple scattering
  events, the proposed approach regularizes the full interior
  frequency domain problem---which, if obtained by either Fourier or
  Laplace transformation of the corresponding interior time-domain
  problem, must encapsulate infinitely many scattering events, giving
  rise to non-uniqueness and eigenfunctions in the Fourier case, and
  ill conditioning in the Laplace case. Numerical examples are included which demonstrate the accuracy and  efficiency of the proposed methodology.
\end{abstract}

\section{Introduction}
\label{sec:1}

The numerical solution of the classical scalar second-order wave
equation remains a challenging problem, with significant impact,
directly and indirectly, on the simulation of propagation and
scattering of time dependent acoustic, elastic and electromagnetic
waves. Methods often utilized in both the literature and applications,
such as the finite difference method~\cite{T00}, the finite element \textcolor{r1}{(FE)}
method~\cite{FP96,LLC97,YBL20} and the discontinuous Galerkin \textcolor{r1}{(DG)}
method~\cite{GSS06,XCS13}, rely on use of volumetric
discretizations of the spatial domain in conjunction with appropriate
time-stepping discretization methods\textcolor{r1}{; recent related
  contributions include the unconditionally stable space-time FE/DG
  methods~\cite{BMPS21,LSZ21} which can avoid use of fine temporal meshes even in the common situations in which fine spatial meshes are required for resolution of challenging geometric features. Volumetric} discretization
approaches can treat problems in general geometries and including
spatially varying media. As is well known, however, \textcolor{r1}{such methods
often} suffer from spatial and temporal numerical dispersion errors
(also known as pollution errors~\cite{BS97,LZZ20}), and they therefore
require use of fine spatial and temporal meshes---and thus, large
computer-memory and run-times---to achieve accurate solutions in
applications involving high frequencies and/or long time simulations.

The time-domain boundary integral equation method (TDBIE) for the wave
equation, which, based on use of the retarded-potential Green's
function, only requires discretization of lower-dimensional domain
boundary, has attracted attention
\textcolor{r1}{recently~\cite{ADG11,Yilmaz:04,BGH20,D03,Betcke:17,SU22,SUZ21}}. This
method requires treatment of the Dirac delta function, and it
therefore leads to integration domains given by the intersection of
the light cone with the overall scattering surface. As a result, the
schemes resulting from the discretization of the TDBIE are generally
complex, and, additionally, they have presented challenges concerning
numerical stability~\cite{BGH20}.  The ``Convolution Quadrature''
method~\cite{BK14,L94,S16,Betcke:17} (CQ), in turn, relies on the
combination of a finite-difference time discretization and a
Laplace-like transformation to reduce the time-domain problem to
modified Helmholtz problems over a range of frequencies. The Helmholtz
problems in the CQ context are tackled by means of frequency-domain
integral equations, and, thus, the CQ method effectively eliminates
spatial dispersion. The solution method inherits the dispersive
character of the finite-difference approximation that underlies the
time-domain scheme, however. A certain ``infinite tail'' in the CQ
time history that results from ``the passage through the Laplace
domain'' also presents ``a serious
disadvantage''~\cite[Chapter~5.1]{S16}.

A frequency-time hybrid solver has recently been
proposed~\cite{ABL18}, in which the time evolution is evaluated by
means of a certain ``windowing and time-recentering'' procedure. The
algorithm presented in that reference simply decomposes the incident
time signal as a sum of a sequence of smooth compactly supported
incident ``wave packets''. Using Fourier transformation in time, the
solution for each one of the wave packets is expressed in terms of
regular-Helmholtz frequency-domain solutions---thus eliminating
spatial dispersion, just like the CQ method. The ``recentering''
strategy then allows for use of a fixed set of frequency-domain
solutions for arbitrarily long times. A tracking strategy is used to
determine the time interval during which the solution associated with
each wave packet must be kept as part of the simulation. An efficient
implementation of the required Fourier transformation processes is
introduced in~\cite{ABL18} which includes specialized high-frequency
algorithms, including, e.g.  ``time re-centering'' of the wave as well
as Chebyshev and \textcolor{r1}{Fourier-Continuation} Fourier transform
representations.  Unlike other approaches, the hybrid
method~\cite{ABL18} \textcolor{r1}{can provide highly accurate
  numerical solutions for problems involving complex scatterers for
  incident fields applied over long periods of time}. The method
allows for time leaping, parallel-in-time implementation and,
importantly, spectral accuracy in time. The CQ approach and other
hybrid methods~\cite{DSSB93,MMRZ00}, in contrast, have only provided
solutions for incident signals of very brief time duration, as
indicated in the various comparisons with other methods provided
in~\cite{ABL18}.

The present paper proposes an extension of the time-domain
method~\cite{ABL18} to problems posed in interior physical domains. An
immediate challenge arises as such a program is contemplated, namely,
that the interior-domain Helmholtz equation is not uniquely solvable
at any frequency whose negative square is an eigenvalue of the Laplace
operator. This problem does not arise if the CQ method is used
instead: the resulting modified Helmholtz problems are uniquely
solvable for all Laplace frequencies. In order to avoid the
aforementioned time-dispersion and infinite tail difficulties inherent
in the CQ method, however, the present paper retains the use of the
Fourier transform, and it re-expresses the full time-evolution as a
problem of multiple scattering among various portions of the domain
boundary. Thus, taking into account the wave's finite speed of
propagation, the original domain boundary is decomposed into a number
$N_\mathrm{arc}$ of {\em overlapping} open-arcs, each one of which
gives rise to a corresponding scattering problem, in absence of all
other arcs in the decomposition. \textcolor{r1}{For simplicity, this paper restricts
attention to the case $N_\mathrm{arc} = 2$, but a numerical
illustration is presented for a simple $N_\mathrm{arc} > 2$
case. (Complete details concerning the algorithm and its
implementation for arbitrary $N_\mathrm{arc} \geq 2$ will be presented
elsewhere~\cite{BBY}.)} In view of \textcolor{r1}{Theorems~\ref{equivalence} and \ref{equivalence11}} below (see
also Algorithms~\ref{alg0}-\ref{alg2}), by appropriately accounting
for multiple scattering, solutions for such open-arc problems can be
combined into a full solution, which is mathematically exact and
numerically accurate, for the given interior domain
problem. Crucially, the frequency-domain open-arc scattering problems
that result upon Fourier transformation are uniquely solvable.

Solution via e.g. a Laplace transformation in time,
  in contrast, while also eliminating the difficulties arising from
  the existence of interior-eigenvalues (and associated lack of
  existence and uniqueness for the necessary frequency-domain
  problems), entails the instability inherent in numerical inverse
  Laplace transformation. We suggest that this Laplace-transform
  instability reflects precisely the use of frequency-domain solutions
  that incorporate infinitely many multiple scattering events, from
  which the solution up to a given finite time $T$ is then to be
  obtained---somehow eliminating, via high-frequency cancellations,
  all contributions from multiple scattering events beyond time $T$,
  and thereby, in view of such cancellations, incorporating a powerful
  source of ill conditioning at any finite spatio-temporal
  discretization level. Note that each frequency-domain solution in
  the Laplace frequency domain indeed contains infinite-time
  information, as is evidenced by the fact that the same set of
  frequency-domain solutions can theoretically be used to propagate
  the time-domain solution of the wave equation up to arbitrarily long
  times. The proposed multiple-scattering algorithm avoids the
  instability by restricting the number of multiple scattering events
  considered to what is strictly necessary to advance the solution up
  to a given finite time.

In the proposed algorithm the necessary frequency-domain open-arc
scattering problems are obtained by means of a frequency-domain integral
equation solver, as indicated in Algorithm~\ref{alg1}. In view of the
classical regularity theory for open-surface problems \textcolor{r1}{(see~\cite{BL12,CDD03,LB15,SW84,PS88})}, the open arc solutions are singular at the arc endpoints: they
behave like a non-integer power of the distance to the endpoint and, e.g., in the case of Dirichlet boundary conditions considered in this paper, they
tend to infinity as the endpoint is approached. The two-dimensional
version~\cite{BY21} of the Chebyshev-based rectangular-polar
discretization methodology~\cite{BG18}, which incorporates a change of
variables introduced in \cite[Eq. (4.12)]{BY20}, is utilized to
evaluate the corresponding integrals with a high order of accuracy. Together with an
appropriate geometrical description, such as those provided by
engineering NURBS-based models---which include parametrizations
expressed in terms of certain types of Rational B-Splines---the
overlapping-patch boundary-partitioning strategy can be used to tackle
interior wave-equation problems in general three-dimensional
engineering structures. Such extensions of the proposed methods,
however, are not considered in this paper, and are left for future
work.

This paper is organized as follows. Section~\ref{sec:2.1} describes
the wave propagation and scattering problem under
consideration. Section~\ref{sec:2.2} introduces the overlapping-arc
scattering structure, and the time-domain boundary integral equations
for the open-arc time-domain scattering problems. A necessary
Huygens-like domain-of-influence condition is introduced in
Section~\ref{sec:2.3}, which simply states that, as in free space,
waves move along boundaries at the speed of sound. Surprisingly, to
the best of the authors knowledge, such a result has not been
established as yet. A discussion in this regard is presented
in \textcolor{r1}{Section}~\ref{sec:2.3}, including a rigorous proof of validity
\textcolor{r1}{in a simple geometrical context as well as} clear numerical evidence of validity in other cases; the rigorous proof of
validity of this condition for general curves is left for future
work. As a byproduct of the constructions concerning the Huygens
condition, a 2D double-layer time-domain formulation is introduced in
Remark~\ref{double_layer} in Appendix~\ref{sec:A} which bypasses
certain difficulties encountered previously. On the basis of these
materials, Section~\ref{sec:2.4} re-expresses the interior time-domain
problem in terms of a proposed open-arc ``ping-pong''
multiple-scattering approach, and it presents the main theoretical
result of this paper, Theorem~\ref{equivalence}---which establishes
that the interior time-domain problem is indeed equivalent to the
proposed ping-pong problem. Section~\ref{sec:2.5} then re-expresses
the ping-pong problem in terms of associated open-arc frequency-domain
problems, and Section~\ref{sec:2.6} presents the \textcolor{r1}{aforementioned}
windowing and time-recentering strategy that is used to enable the
treatment of problems of arbitrary long time duration. The numerical
implementation of the multiple scattering approach is presented in
Section~\ref{sec:3}, including the windowed Fourier-transform
algorithm~\cite{ABL18} used (Section~\ref{sec:3.1}),
\textcolor{r1}{the methods utilized for the evaluation of singular frequency-domain integral operators, and a novel arc-extension approach that facilitates the avoidance of open-arc endpoint singularities (Section~\ref{sec:3.2}). The} overall computational implementation is outlined in
Section~\ref{sec:3.3}. Numerical examples demonstrating the accuracy
and efficiency of the proposed approach, finally, are presented in
Section~\ref{sec:4}.

\section{Hybrid frequency-time multiple scattering interior solver}
\label{sec:2}

\subsection{Wave equation problem}
\label{sec:2.1}

Let $\Omega\subset\R^2$ denote a bounded domain with piecewise smooth
boundary $\Gamma=\partial\Omega$, and let $u^i(x,t)$ denote a given incident field defined for $t\in\R$ and $x\in \Gamma$, which
vanishes for $t\le 0$. In what follows we consider the wave equation
initial and boundary-value problem
\begin{eqnarray}
\label{waveeqn}
\begin{cases}
\frac{\pa^2u^s}{\pa t^2}(x,t)-c^2\Delta u^s(x,t)=0, & (x,t)\in\Omega\times \textcolor{r1}{\R_+}, \cr
\textcolor{r1}{u^s(x,0)=\frac{\pa u^s}{\pa t}(x,0) =0,} & \textcolor{r1}{x\in\Omega}, \cr
u^s(x,t)=- u^i(x,t),  & (x,t)\in\Gamma\times\textcolor{r1}{\R_+},
\end{cases}
\end{eqnarray}
for the scattered field $u^s(x,t)$ throughout $\Omega$, where the
constant $c>0$ denotes the wave-speed \textcolor{r1}{and $\R_+:=\{t\in\R: t>0\}$. Since
  $u^i(x,t)=0$ for $t<0$, problem (\ref{waveeqn}) can be equivalently written in the form
\begin{eqnarray}
\label{waveeqn1}
\begin{cases}
\frac{\pa^2u^s}{\pa t^2}(x,t)-c^2\Delta u^s(x,t)=0, & (x,t)\in\Omega\times \textcolor{r1}{\R}, \cr
u^s(x,t)=- u^i(x,t),  & (x,t)\in\Gamma\times\textcolor{r1}{\R},
\end{cases}
\end{eqnarray}
by invoking the causality condition $u^s(x,t)=0, x\in\Omega,
t\le0$. Throughout this paper a number of wave-equation problems will
be considered which, assuming vanishing boundary data for $t\leq 0$,
will be expressed in a form similar to (\ref{waveeqn1}), without
explicit mention of vanishing initial conditions at $t=0$.  The
well-posedness of the wave equation problem (\ref{waveeqn}) (and,
equivalently, (\ref{waveeqn1})) in an appropriate Sobolev space is
addressed in Theorem~\ref{welltime} below.}

As indicated in Section~\ref{sec:1}, this paper proposes a fast hybrid
method, related to that presented in~\cite{ABL18}, for the numerical
solution of this problem. As noted in that section, however, the
frequency-domain solutions required by the hybrid method~\cite{ABL18}
fail to exist, in the present interior-domain context, at frequencies
corresponding to Laplace eigenvalues in the domain $\Omega$---and,
thus, the exterior-domain hybrid approach~\cite{ABL18} does not apply
in the present interior-domain \textcolor{r1}{setting}. The hybrid
approach proposed in the present paper relies, instead, on a multiple
scattering strategy that transforms the original wave equation problem
in a bounded domain into a sequence of wave equation problems of
scattering by overlapping {\em open-arcs}---for which the frequency
domain solutions exist at all frequencies, and for which, therefore,
general time-domain solutions can effectively be obtained via the
windowing and recentering Fourier-transform methods introduced
in~\cite{ABL18}. The overlapping-arc scattering structure used as well
as necessary theoretical results concerning open-arc time-domain
scattering problems are presented in the following \textcolor{r1}{sections}.

\begin{figure}[htb]
\centering
\begin{tabular}{ccc}
\includegraphics[scale=0.4]{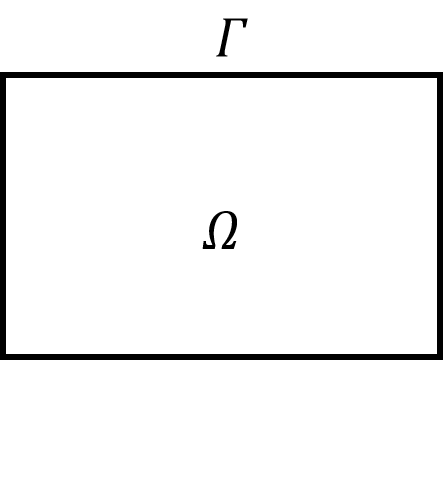} &
\includegraphics[scale=0.4]{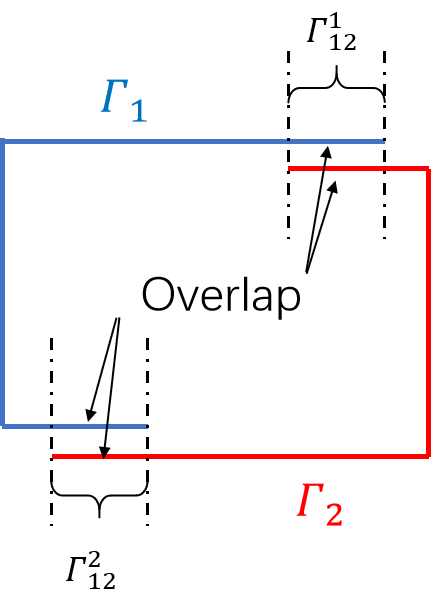} &
\includegraphics[scale=0.4]{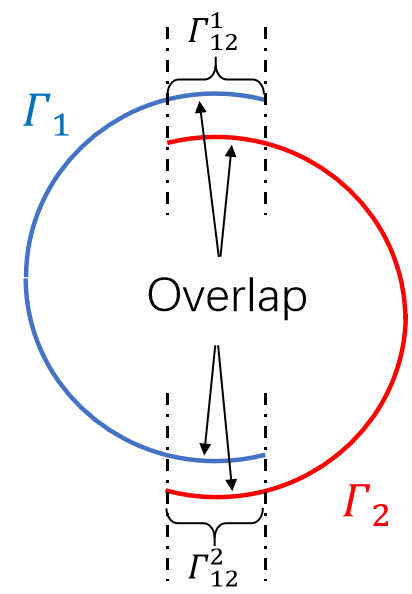} \\
(a) & (b) & (c)
\end{tabular}
\caption{Decomposition of closed boundaries $\Gamma$ into pairs of
  overlapping open arcs. (b) Decomposition of the rectangular boundary $\Gamma$ depicted in (a); (c) Decomposition of a circular closed curve $\Gamma$.}
\label{MSmodel}
\end{figure}

\subsection{Overlapping-arc geometry and time-domain boundary integral equations}
\label{sec:2.2}

The proposed multiple scattering algorithm relies on a boundary
decomposition strategy based on use of overlapping patches. While the
algorithm can utilize an arbitrary numbers of patches, for simplicity,
this paper only considers decompositions consisting of two
patches\textcolor{r1}{, but an algorithm based on an arbitrary number
  of patches can be constructed. (This is illustrated in Section 4 by
  means of a numerical example wherein three patches are used; full
  details concerning the multi-patch algorithm and its implementation
  will be presented elsewhere \cite{BBY}.)} Thus, as illustrated in
Figure~\ref{MSmodel}(b), the scattering surface $\Gamma$ in
Figure~\ref{MSmodel}(a) is covered by two overlapping patches
\textcolor{r1}{$\Gamma_1=\ov{\Gamma}_1\subset\Gamma$ and
  $\Gamma_2=\ov{\Gamma}_2\subset\Gamma$},
$\Gamma=\Gamma_1\cup\Gamma_2 $, whose intersection
$\Gamma_{12}=\Gamma_1\cap\Gamma_2$ equals the disjoint union
$\Gamma_{12}=\Gamma_{12}^{1}\cup\Gamma_{12}^{2}$ of two connected
components $\Gamma_{12}^{1}$ and $\Gamma_{12}^{2}$. (Here the overline
denotes the closure of the corresponding set.) A similar decomposition
is presented in Figure~\ref{MSmodel}(c) for a different \textcolor{r1}{curve}
$\Gamma$. The ``truncation'' of $\Gamma_1$ and $\Gamma_2$ by
$\Gamma_{12}$ results in the truncated arcs
$\Gamma_j^{\rm tr}=\Gamma_j\backslash\Gamma_{12}$, $j=1,2$.  The
distance between the arcs $\Gamma_1^{\rm tr}$ and $\Gamma_2^{\rm tr}$,
which is denoted by
\begin{equation}\label{dist}
  \delta_{12}=\mbox{dist}\{\Gamma_1^{\rm tr},\Gamma_2^{\rm tr}\},
\end{equation}
plays an essential role in our theory and algorithms.

Let us now consider the unbounded domains
$\Omega_j = \R^2\backslash \Gamma_j$ ($j=1,2$) and the corresponding
time-domain problems of scattering by the arcs $\Gamma_j$, which \textcolor{r1}{
underly} the proposed multiple scattering solution strategy. Given an
incident signal $g_j(x,t)$ defined for $x\in\Gamma_j$ ($j=1,2$) and
$t\in\mathbb{R}$, which vanishes for $t\leq 0$, we consider the
following wave equation problem for the function $w_j^s(x,t)$:
\textcolor{r1}{
  \begin{eqnarray}
\label{eq_op}
\begin{cases}
\frac{\pa^2w^s_{j}}{\pa t^2}(x,t)-c^2\Delta
w^s_{j}(x,t)=0, &(x,t)\in\Omega_j\times \R, \cr
w_j^s(x,t)=g_j(x,t),  & (x,t)\in\Gamma_j\times\R.
\end{cases}
\end{eqnarray}
} As is well known~\cite{S16}, the solution $w_j^s(x,t)$ admits the
single-layer representation \begin{eqnarray}
\label{solreptime}
w_j^s(x,t)=\widetilde{\mathcal{S}}_j[\widetilde{\psi}_j](x,t), \quad
x\in\Omega_j, \end{eqnarray} where $\widetilde{\psi}_{j}$ is the solution of the
time-domain integral equation \begin{eqnarray}
\label{BIEtime}
\widetilde{S}_j[\widetilde{\psi}_{j}]=g_{j}\quad\mbox{on}\quad
\Gamma_j.  \end{eqnarray} Here the time-domain single-layer potential
$\widetilde{\mathcal{S}}_j$ is defined by
\begin{eqnarray}
\label{operatorS}
\widetilde{\mathcal{S}}_j[\widetilde{\psi}_{j}](x,t)= \frac{1}{2\pi} \int_{\Gamma_j}\int_0^{t-c^{-1}|x-y|} \frac{\widetilde{\psi}_{j}(y,\tau)} {\sqrt{(t-\tau)^2-c^{-2}|x-y|^2}} d\tau ds_y, \quad \textcolor{r1}{x\in\Omega_j},
\end{eqnarray}
and \textcolor{r1}{the time-domain single-layer boundary integral operator $\widetilde{S}_j:=\gamma_j\widetilde{\mathcal{S}}_j$ is given by}
\begin{equation}\label{operatorS_bdry}
\textcolor{r1}{\widetilde{S}_j[\widetilde{\psi}_{j}](x,t)=\frac{1}{2\pi} \int_{\Gamma_j}\int_0^{t-c^{-1}|x-y|} \frac{\widetilde{\psi}_{j}(y,\tau)} {\sqrt{(t-\tau)^2-c^{-2}|x-y|^2}} d\tau ds_y, \quad x\in\Gamma_j,}
\end{equation}
\textcolor{r1}{where
  $\gamma_j: H_{\sigma,\alpha}^{p}(\R,H^1(\Omega_j))\rightarrow
  H_{\sigma,\alpha}^{p}(\R,H^{1/2}(\Gamma_j))$ denotes the trace
  operator. Here, for given $\sigma>0$ and $\alpha,p\in\R$, and for a
  given Hilbert space $D$, we have used the spatio-temporal Sobolev
  spaces $H_{\sigma,\alpha}^p(\R,D)$ of functions with values in $D$
  which vanish for $t\leq \alpha$. The spaces
  $H_{\sigma,\alpha}^p(\R,D)$ are defined by~\cite{BH86,CHLM10}
\begin{equation}
H_{\sigma,\alpha}^p(\R,D):=\left\{ f\in\mathcal{L}'_{\sigma,\alpha}(D): \int_{-\infty+i\sigma}^{\infty+i\sigma} |s|^{2p}\|\mathcal{L}[f](s)\|_D^2ds<\infty \right\}
\end{equation}
together with the norm
\begin{equation}
\label{norm}
\|f\|_{H_{\sigma,\alpha}^p(\R,D)}:=\left(\int_{-\infty+i\sigma}^{\infty+i\sigma} |s|^{2p}\|\mathcal{L}[f](s)\|_D^2ds \right)^{1/2},
\end{equation}
where $\mathcal{L}[f]$ denotes the Fourier-Laplace transform of $f$ given by
\begin{equation}
\mathcal{L}[f](s):=\int_{-\infty}^\infty f(t)e^{is t}dt,\quad s\in \C_\sigma:=\{\omega\in\C: \mbox{Im}(s)>\sigma>0\},
\end{equation}
and where
$\mathcal{L}'_{\sigma,\alpha}(D):=\{\phi\in\mathcal{D}_\alpha'(D):
e^{-\sigma t}\phi\in \mathcal{S}_\alpha'(D)\}$ is defined in terms of
the sets $\mathcal{D}_\alpha'(D)$ and $\mathcal{S}_\alpha'(D)$ of
$D$-valued distributions and $D$-valued tempered distributions that
vanish for $t\le \alpha$, respectively. We also call \ben
H_{\alpha}^p(\beta,D)=\{f(x,t)|_{t\in(-\infty,\beta]}: f\in
H_{\sigma,\alpha}^p(\R,D)\} \enn the set of all restrictions of
functions $f\in H_{\sigma,\alpha}^p(\R,D)$ to the interval
$-\infty < t\leq \beta$. It can be easily checked that, as suggested
by the notation used, the space $H_{\alpha}^p(\beta,D)$ does not
depend on $\sigma$. This can be verified for integer values of $p$ by
using a norm equivalent to (\ref{norm}) that is expressed in terms of
derivatives with respect to the variable $t$, in conjunction with
smooth and compactly-supported window functions of $t$ which equals one
over the restriction interval $(-\infty, \beta]$ for a given value of
$\beta$. The equivalence for non-integer values of $p$ follows by
interpolation.}

\textcolor{r1}{The well-posedness of the wave
  equation problems (\ref{waveeqn}) and \eqref{eq_op}} is \textcolor{r1}{established in}
the following theorem~\cite{CHLM10,YMC22}.
\begin{theorem}
\label{welltime}
\textcolor{r1}{For given $p\in\R$, $\alpha\ge0$ and for $j=1,2$ we have:
  \begin{itemize}
  \item[(a)] Given $u^i\in H_{\sigma,\alpha}^p(\R,H^{1/2}(\Gamma))$, the wave
    equation problem (\ref{waveeqn1}) admits a unique solution
    $u^s\in H_{\sigma,\alpha}^{p-3/2}(\R,H^1(\Omega))$.
  \item[(b)] Given $g_j\in H_{\sigma,\alpha}^p(\R,H^{1/2}(\Gamma_j))$, the
    wave equation problem \eqref{eq_op} admits a unique solution
    $w_j^s\in H_{\sigma,\alpha}^{p-3}(\R,H^1(\Omega_j))$.
\end{itemize}
}

\end{theorem}

\subsection{Huygens-like domain-of-influence along boundaries}
\label{sec:2.3}

The multiple scattering algorithm proposed in this paper depends in an
essential manner on a certain domain-of-influence condition, stated as
Condition~\ref{localpro0} below, which is in essence a variant of the
well known Huygens principle in a form that is applicable to the
problem of scattering by obstacles and open arcs. Thus,
Condition~\ref{localpro0} expresses a well accepted principle in wave
physics, namely, that solutions of the wave equation propagate at the
speed of sound, and that the wave field vanishes identically before
the arrival of a wavefront. This property has been rigorously
established by the method of spherical means~\cite{BC39} for the
problem of propagation of waves in space without scatterers. Further,
some mathematical results have previously been given for the
corresponding problem of scattering by obstacles~\cite[Proposition
3.6.2]{S16}. But previously available results for obstacle-scattering
problems are not sharp, as they only ensure that the field propagates
away from the complete boundary (with speed equal to the speed of
sound), but they do not account for propagation {\em along} the
scattering boundary. In particular, for incident fields illuminating a
subset of the boundary of a scatterer, previous theoretical results do
not establish that the field propagates at the speed of sound along
the scattering boundary. This boundary-propagation character provides
a crucial element in the main theorem of this paper,
Theorem~\ref{equivalence}---which, showing that the exact solution of
the problem (\ref{waveeqn}) can be expressed as the sum of a series of
multiple-scattering iterates, forms the basis of the ping-pong
multiple-scattering algorithm proposed in this paper. Although we
conjecture that Condition~\ref{localpro0} is always valid, to the best
of our knowledge such a result has not previously been
established. \textcolor{r1}{A full theoretical treatment of this
  problem is beyond the scope of this paper, but, as indicated in
  Remark~\ref{rem_conj}, this paper does include a complete proof for
  the case of straight arcs as well as clear numerical supporting
  evidence for the validity of this condition for curved arcs.}

\textcolor{r1}{Let $\mathcal{C}$ denote an Lipschitz open arc. Given
  an incident signal $g(x,t)$ defined for $x\in\mathcal{C}$ and
  $t\in\mathbb{R}$, which vanishes for $t\leq 0$, consider the wave
  equation problem
  \begin{eqnarray}
\label{eq_opp}
\begin{cases}
\frac{\pa^2w^s}{\pa t^2}(x,t)-c^2\Delta
w^s(x,t)=0, &(x,t)\in \R^2\backslash\mathcal{C}\times \R, \cr
w^s(x,t)=g(x,t),  & (x,t)\in \mathcal{C}\times\R.
\end{cases}
\end{eqnarray}
Using these notations, the necessary Huygens-like condition is
presented in what follows.}

\begin{condition}
\label{localpro0}
\textcolor{r1}{We say that an open Lipschitz curve $\mathcal{C}$ with
  endpoints $e_1$ and $e_2$ satisfies the restricted Huygens condition iff for
  every Lipschitz curve
  $\mathcal{C}^{\mathrm{inc}}\subseteq\mathcal{C}$ satisfying
  $\mathrm{dist}(\mathcal{C}^{\mathrm{inc}},\{e_1,e_2\})> 0$, and for
  every function $g\in H_{\sigma,0}^{p}(\R,H^{1/2}(\mathcal{C}))$ defined in $\mathcal{C}$ such that} \be
\label{DoI_ball_ass}
\textcolor{r1}{\{x\in\mathcal{C}\ |\ g(x,t)\ne 0 \}\subseteq \mathcal{C}^{\mathrm{inc}}\quad \mbox{for all}\quad t>0,}
\en
\textcolor{r1}{we have}
\be
\label{DoI_ball_res1}
\textcolor{r1}{\{x\in\R^2 \ |\ w^s(x,t)\ne 0\} \subseteq \Lambda^s(t)\quad \mbox{for all}\quad t\le c^{-1}\mathrm{dist}(\mathcal{C}^{\mathrm{inc}},\{e_1,e_2\}),}
\en
where
\ben
\textcolor{r1}{\Lambda^s(t)=\{x\in\R^2 \ |\ \mathrm{dist}(x,\mathcal{C}^{\mathrm{inc}})\leq ct\}.}
\enn
\end{condition}

\begin{remark}\label{rem_conj}
  \textcolor{r1}{We conjecture that Condition~\ref{localpro0} holds
    for arbitrary open and closed Lipschitz curves $\mathcal{C}$
    (where, in the case of closed curves, the wave equation problem is
    posed either in the interior or the exterior of the curve) and for
    all $t>0$ (without the restriction
    $t\le
    c^{-1}\mathrm{dist}(\mathcal{C}^{\mathrm{inc}},\{e_1,e_2\})$). The
    proof is left for future work. The validity of
    Condition~\ref{localpro0} for the case in which $\mathcal{C}$ is a
    line segment is established in the following lemma. We have also
    verified numerically the validity of this condition for a wide
    range of curved open arcs; one such verification is presented in
    Section 2.3.1 below.}
\end{remark}

\textcolor{r1}{In what follows we denote
  $\R^2_\pm:=\{x=(x_1,x_2)\in\R^2: x_2\gtrless0\}$ and
  $\R^2_0:=\{x=(x_1,x_2)^\top\in\R^2: x_2=0\}$.}
\begin{lemma}
\label{lemma_circle}
\textcolor{r1}{Let $c_1<c_2$. Then the (straight) open arc
  $\mathcal{C}=(c_1,c_2)\times\{0\}\subset\R^2_0$ satisfies the restricted
  Huygens Condition~\ref{localpro0}.}
\end{lemma}
\begin{proof}
  \textcolor{r1}{Let $\mathcal{C}^{\mathrm{inc}}\subseteq \mathcal{C}$
    denote an arc contained in $\mathcal{C}$ satisfying
    $\mathrm{dist}(\mathcal{C}^{\mathrm{inc}},\{(c_1,0),(c_2,0)\})>
    0$, let a function
    $g\in H_{\sigma,0}^{p}(\R,H^{1/2}(\mathcal{C}))$ be given that
    satisfies the assumption (\ref{DoI_ball_ass}), extend $g$ to all
    of $\R^2_0$ by setting $g=0$ in $\R^2_0\backslash\mathcal{C}$, and
    consider the problems \be
\label{waveeqn-ballTD}
\begin{cases}
\frac{\pa^2v_\pm^s}{\pa t^2}(x,t)-c^2\Delta v_\pm^s(x,t)=0, & (x,t)\in \R^2_\pm\times\R, \cr
v_\pm^s(x,t)= g(x,t),  & (x,t)\in\R^2_0\times\R
\end{cases}
\en for the functions $v_\pm^s = v_\pm^s(x,t)$. In view of equation
(\ref{appendx_solrep}) in Appendix~\ref{sec:A} it follows that
\begin{align}
\label{TDsol1}
 v_+^s&(x,t)=\frac{1}{\pi c^2} \nonumber\\
& \int_{\mathcal{C}^{\mathrm{inc}}}\int_{0}^{t-c^{-1}|x-y|} \left[\frac{x_2g(y,\tau)} {(t-\tau)^2\sqrt{(t-\tau)^2-c^{-2}|x-y|^2}} +\frac{x_2g^{(1)}(y,\tau)} {(t-\tau)\sqrt{(t-\tau)^2-c^{-2}|x-y|^2}} \right] d \tau ds_y
\end{align}
for all $x\in\R^2_+$ and all $t>0$, where
$g^{(1)}(x,t)=\frac{\pa g}{\pa t}(x,t)$. Noting that for $t>0$ and
$x\notin\Lambda^s(t)$ we have $t-c^{-1}|x-y|<0$ for all
$y\in\mathcal{C}^{\mathrm{inc}}$, and since $g(\cdot,t)=0$ for all
$t\le0$ by assumption, we conclude that
$\{x\in\R^2 \ |\ v_+^s(x,t)\ne 0\} \subseteq \Lambda^s(t)$ for all
$t>0$. Similarly,
$\{x\in\R^2 \ |\ v_-^s(x,t)\ne 0\} \subseteq \Lambda^s(t)$ for all
$t>0$. It follows that \ben
\begin{cases}
v_+^s(x,t)=v_-^s(x,t) =0 \cr
\pa_{x_2}v_+^s(x,t)=\pa_{x_2}v_-^s(x,t) =0
\end{cases}
\quad\mbox{for}\quad x\in \R_0^2\backslash\mathcal{C}\quad\mathrm{and}\quad t\le c^{-1}\mathrm{dist}(\mathcal{C}^{\mathrm{inc}},\{(c_1,0),(c_2,0)\}).
\enn
This implies  that
\ben
w^s(x,t)=\begin{cases}
v_+^s(x,t), & x\in\R^2_+,\cr
v_-^s(x,t), & x\in\R^2_-,\cr
0, & x\in \R^2_0\backslash\mathcal{C},\cr
\end{cases}\quad t\le c^{-1}\mathrm{dist}(\mathcal{C}^{\mathrm{inc}},\{(c_1,0),(c_2,0)\}),
\enn is the unique solution to the wave equation
problem~\eqref{eq_opp} for
$t\le c^{-1}\mathrm{dist}(\mathcal{C}^{\mathrm{inc}},\{(c_1,0),(c_2,0)\})$. Hence,
the condition \ben \{x\in\R^2 \ |\ w^s(x,t)\ne 0\} \subseteq
\Lambda^s(t)\quad \mbox{for all}\quad t\le
c^{-1}\mathrm{dist}(\mathcal{C}^{\mathrm{inc}},\{(c_1,0),(c_2,0)\})
\enn for the function $w^s$ follows from the corresponding properties,
established above, for the functions $v_\pm^s$, and the proof of the
lemma is complete.}
\end{proof}

\textcolor{r1}{For ease of reference, in the following lemma we
  present the Huygens Condition~\ref{localpro0} in the form that will
be used in the proof of Theorem~\ref{equivalence}. In order to match
the setting of the theorem, for an integer $j$ we introduce the
notation
\begin{equation}\label{prime} j' = \mathrm{mod}(j,2)+1\quad
  (j\in\mathbb{N}),
\end{equation}
where, for integers $a$ and $b$, $\mathrm{mod}(a,b)$ denotes the remainder of the division of $a$
by $b$. In our context, where the index values $j=1,2$ refer to the
corresponding arcs $\Gamma_1$, $\Gamma_2$, we have $j'=1$
(resp. $j'=2$) for $j=2$ (resp. $j=1$).
\begin{lemma}
\label{localpro}
Let $j\in\{1,2\}$, $p\in\R$, and $T_0>0$, and assume that, for
$j\in\{1,2\}$, (a)~$g_j\in H_{\sigma,T_0}^{p}(\R,H^{1/2}(\Gamma_j))$
satisfies
\begin{equation}\label{ini_set}
  g_j(x,t)=0\quad\mbox{for}\quad
  (x,t)\in\Gamma_{12}\times \R;
\end{equation}
and, (b)~$\Gamma_j$ satisfies Condition~\ref{localpro0}. Then,
recalling equation~\eqref{dist}, letting $t_0=\delta_{12}/c>0$, and
calling
\textcolor{r1}{$w_j^s\in H_{\sigma,T_0}^{p-3}(\R,H^1(\Omega_j))$ the
  unique} solution of the wave equation problem \eqref{eq_op}, we have
\begin{equation}\label{ini_set_res}
  w_j^s(x,t)=0\quad\mbox{for}\quad
  (x,t)\in \Gamma_{j'}^{\rm tr} \times (-\infty,T_0+t_0].
\end{equation}
\end{lemma}}

\subsubsection{Numerical verification of \textcolor{r1}{the Huygens condition for
 elliptical arcs}.\label{sec:2.3.1}}
As indicated above, we have conducted a number of numerical tests
\textcolor{r1}{which clearly suggest that}, as expected,
Condition~\ref{localpro0} and Lemma~\ref{localpro} are universally
valid. For reference \textcolor{r1}{in this section we present the results of one such test}. To
introduce our example we let \ben
&& \Gamma_1=\{x=(\cos\theta,1.5\sin\theta): \theta\in(0.5\pi,1.5\pi)\},\\
&& \Gamma_2^{\mathrm{tr}}=\{x=(\cos\theta,1.5\sin\theta): \theta\in(-0.5\pi,0.5\pi)\},\\
&& \Gamma_{12}=\{x=(\cos\theta,1.5\sin\theta):
\theta\in(0.5\pi,0.75\pi)\cup(1.25\pi,1.5\pi)\}, \enn and we consider
the wave equation problem \eqref{eq_op} with $j=1$ and \ben
g_1(x,t)=\begin{cases} \left[1-\cos4(\theta-0.75\pi)\right]
  \exp(-16(t-3)^2), &
  (x,t)=(\cos\theta,1.5\sin\theta)\in\Gamma_1^{\mathrm{tr}}\times
  (T_0,\infty), \cr 0, & (x,t)\in\Gamma_{12}\times (T_0,\infty),\cr 0,
  & (x,t)\in\Gamma_1\times (-\infty,T_0].
\end{cases}
\enn
where $T_0=1.74$. (This selection of $T_0$ makes $g_1(x,t)$ ``approximately continuous'' at $t=T_0$, since, as is easily checked, $\left[1-\cos4(\theta-0.75\pi)\right] \exp(-16(t-3)^2)<10^{-11}$ for $(x,t)\in\Gamma_1^{\mathrm{tr}}\times (-\infty,T_0)$.) It can also be checked that $t_0=\delta_{12}/c\approx 0.83$ for the geometry under consideration.

Table~\ref{testhuygens1} presents the maximum values of $|w_1^s(x,t)|$
over several time intervals at four points on $\Gamma_2^{\mathrm{tr}}$
\ben x_1=(0.031,1.499),\quad x_2=(0.5,1.299),\quad x_3=(0.866,0.75),
\quad x_4=(1,0); \enn note, in particular, that the point
$(0.031,1.499)$ is very close to $\Gamma_1$. The first column in this
table shows that for all four points \textcolor{r1}{$x_j\in\Gamma_2^{\mathrm{tr}}$, $j=1,\cdots,4$,} and
for $t<T_0+t_0$, the relation (\ref{ini_set_res}) is verified up to
the numerical error, of order $\mathcal{O}(10^{-11})$, inherent in the
numerical solution used. To further illustrate the validity of the
Huygens condition for this test case, we let
$t_\ell=c^{-1}\mathrm{dist}\{x_\ell,\Gamma_1^{\mathrm{tr}}\}$,
$\ell=2,3,4$ which gives \ben t_2\approx 1.23,\quad t_3\approx
1.60\quad\mathrm{and}\quad t_4= 2.  \enn The maximum values of
$|w_1^s(x_\ell,t)|$ listed in the last three columns in
Table~\ref{testhuygens1}, which correspond to the time intervals
$t\in(0,T_0+t_0+t_\ell),\ell=2,3,4$, illustrate, more generally, the
Huygens-like domain-of-influence property \ben
w_1^s(x,t)=0\quad\mbox{for}\quad (x,t)\in \Gamma_{2}^{\rm tr} \times
(-\infty,T_0+t_x),\quad
t_x=c^{-1}\mathrm{dist}\{x,\Gamma_1^{\mathrm{tr}}\}.  \enn

\begin{table}[htb]
\caption{Maximum values of $|w_1^s(x,t)|$ over various time intervals at four points on $\Gamma_2^{\mathrm{tr}}$.}
\centering
\begin{tabular}{|c|c|c|c|c|c|}
\hline
$x$ & $\max\limits_{t\in(0,T_0+t_0)}|w_1^s(x,t)|$ &  $\max\limits_{t\in(0,T_0+t_0+t_2)}|w_1^s(x,t)|$ & $\max\limits_{t\in(0,T_0+t_0+t_3)}|w_1^s(x,t)|$ & $\max\limits_{t\in(0,T_0+t_0+t_4)}|w_1^s(x,t)|$\\
\hline
$x_1$ & $1.81\times 10^{-12}$  & $2.32\times 10^{-8}$  & $1.71\times 10^{-4}$  & $1.70\times 10^{-2}$ \\
\hline
$x_2$ & $4.59\times 10^{-12}$  & $4.59\times 10^{-12}$ & $1.29\times 10^{-7}$  & $1.49\times 10^{-3}$ \\
\hline
$x_3$ & $5.23\times 10^{-12}$  & $5.23\times 10^{-12}$ & $5.23\times 10^{-12}$ & $7.93\times 10^{-7}$ \\
\hline
$x_4$ & $6.98\times 10^{-12}$  & $7.17\times 10^{-12}$ & $7.39\times 10^{-12}$ & $2.61\times 10^{-11}$ \\
\hline
\end{tabular}
\label{testhuygens1}
\end{table}

\subsection{Two-arc ``ping-pong'' multiple scattering construction}
\label{sec:2.4}

Taking into account the finite propagation speed that characterizes
the solutions of the wave equation, we propose to produce the
time-domain solution of the original problem~\eqref{waveeqn} in the
interior domain $\Omega$ as the sum of ``ping-pong'' wave-equation
solutions \textcolor{r1}{produced under multiple scattering} by the arcs $\Gamma_1$
and $\Gamma_2$. To describe the ping-pong multiple-scattering scheme
we introduce a few useful notations and conventions. We call
\begin{equation}\label{index_shift}
  \quad j(m)=2-{\rm mod}(m,2),\quad m=1,2,3,\dots
\end{equation}
(in other words, $j(m)$ equals $1$ or $2$ depending on whether $m$ is
odd or even, respectively), and, as detailed in
Definition~\ref{inductive}, we inductively define boundary-condition
functions $f_m(x,t)$ ($m\geq 1$) and associated wave-equation
solutions $v_m(x,t)$ ($m\geq 1$), all of which are {\em causal}---that
is to say, they vanish identically for $t\leq 0$.

\begin{definition}\label{inductive}
  For $m\in\mathbb{N}$ we inductively define $v^s_m(x,t)$ as equal to
  the solution $w^s_{j(m)}(x,t)$ of the open-arc
  problem~\eqref{eq_op} with boundary data
  \begin{equation}\label{bdry_m}
    v^s_m(x,t) = f_m(x,t)\quad \mbox{for}\quad (x,t)\in\Gamma_{j(m)}\times \textcolor{r1}{\R},\quad (m \in\mathbb{N}),
  \end{equation}
  where $f_{m}(x,t):\Gamma_{j(m)}\times \textcolor{r1}{\R}\to\mathbb{C}$ denotes the causal
  functions defined inductively via the relations
    \begin{equation}
  \label{criterion0}
  f_1(x,t)= -u^i(x,t)\quad\mbox{on}\quad\Gamma_1,\quad f_2(x,t)= -u^i(x,t)-v^s_1(x,t)\quad\mbox{on}\quad\Gamma_2,
\end{equation}
and,   \begin{equation}
  \label{criterion2}
  f_m(x,t)=-v^s_{m-1}(x,t),\quad\mbox{on}\quad\Gamma_{j(m)}, \quad m \geq 3.
\end{equation}
\end{definition}

\begin{remark}\label{vanish12}
The proposed multiple-scattering strategy relies crucially on the
relations
  \begin{equation}\label{vanishing}
    f_{m}(x,t) = 0\quad \mbox{for}\quad
      (x,t)\in\Gamma_{12}\times \textcolor{r1}{\R}, \quad m\in\mathbb{N},\quad m\geq 2,
\end{equation}
which can easily be established inductively, as indicated in what
follows. Considering first the case $m=2$, in view of
Definition~\ref{inductive}, we have
$v^s_2(x,t)=f_2(x,t)=-u^i(x,t)-v^s_1(x,t)$ on $\Gamma_2$, on one hand,
and $v^s_1(x,t)=-u^i(x,t)$ on $\Gamma_1$, on the other. We conclude
that $v^s_2(x,t)=0$ for $(x,t)\in\Gamma_{12}\times \textcolor{r1}{\R}$,
as desired. The inductive step is equally simple: assuming, for
$\ell\in\mathbb{N}, \ell\ge 2$, that $f_\ell(x,t)$ vanishes for
$(x,t)\in\Gamma_{12}\times \textcolor{r1}{\R}$, and in view
of~(\ref{bdry_m}) and~(\ref{criterion2}), we have
$f_{\ell+1}(x,t)=-v^s_\ell(x,t)=-f_{\ell}(x,t)=0$ for
$(x,t)\in\Gamma_{12}\times \textcolor{r1}{\R}$, and~\eqref{vanishing}
follows.
\end{remark}

The main theorem of this paper, which is presented in what follows,
\textcolor{r1}{shows that the solution $u^s=u^s(x,t)$ of equation~\eqref{waveeqn} can
be produced by means of the} $M$-th order multiple-scattering sum
\begin{equation}\label{mult_scatt}
  u^s_M(x,t):=\sum_{m=1}^M
  v^s_{m}(x,t),
\end{equation}
which includes contributions from the ``ping-pong'' scattering
iterates $v^s_m(x,t)$ with $m=1,\dots, M$.

\begin{theorem}
\label{equivalence}
\textcolor{r1}{Let $M\in\mathbb{N}$, $M\geq 2$, $p\in\R$, and let
  $T = T(M) = (M-1) \delta_{12}/c$. Then, given
  $u^i\in H_{\sigma,0}^{p}(\R,H^{1/2}(\Gamma))$, we have
  $u^s_M\in H_{0}^{p-3/2}(T(M),H^1(\Omega))$ and
  $u^s(x,t) =u^s_M(x,t)$ for all
  $(x,t)\in\Omega\times (-\infty,T(M)]$.}
\end{theorem}
\begin{proof}
  \textcolor{r1}{By construction, for all $m\in\mathbb{N}$ the function
    $v^s_m(x,t)$ satisfies the homogeneous wave equation for
    $(x, t)\in\Omega\times \R$ as well as vanishing boundary
    conditions for $t\leq 0$. Using Theorem~\ref{welltime}(b)
    inductively, it follows that, for all $m\in\mathbb{N}$,
    $f_m\in H_{\sigma,0}^{p-3(m-1)}(\R,H^{1/2}(\Gamma_j))$, and
    $v_m^s\in H_{\sigma,0}^{p-3m}(\R,H^1(\Omega))$. In particular,
    $u^s_M\in H_{\sigma,0}^{p-3M}(\R,H^1(\Omega))$.}

  \textcolor{r1}{To complete the proof of the theorem, it suffices to
  show that the function $u^s_M$ satisfies}
  \begin{equation}\label{bound_0}
    u^s_M(x,t)+u^i(x,t)=0\quad\mbox{for}\quad (x,t)\in\Gamma\times (-\infty,T(M)].
  \end{equation}
  \textcolor{r1}{Indeed, from this relation it follows that
    $u^s-u^s_M$ satisfies trivial boundary condition up to time
    $T(M)$. Then it follows from Theorem~\ref{welltime}(a) that
    $u^s-u^s_M\in H_{\sigma,T(M)}^{p-3M}(\R,H^1(\Omega))$, and, in
    particular, $u^s-u^s_M$, vanishes throughout $\Omega$ for all
    $t\leq T(M)$. In other words,
    $u^s_M|_{(-\infty,T(M)]}=u^s|_{(-\infty,T(M)]}\in
    H_{0}^{p-3M}(T(M),H^1(\Omega))$. But from
    Theorem~\ref{welltime}(a) we also know that
    $u^s\in H_{\sigma,0}^{p-3/2}(\R,H^1(\Omega))$, and, thus,
    $u^s|_{(-\infty,T(M)]}\in H_{0}^{p-3/2}(T(M),H^1(\Omega))$. It
    follows that, as claimed, $u^s_M$ is a solution of the wave
    equation that coincides with $u^s$ up to time $T(M)$ and satisfies
    $u^s_M|_{(-\infty,T(M)]}\in H_{0}^{p-3/2}(T(M),H^1(\Omega))$.}

  The validity of the relation~\eqref{bound_0}, and thus, the proof of
  the theorem, are established in what follows by induction on the
  integer $M$. We  first verify the relation~\eqref{bound_0} in
  the case $M=2$. Since $u^s_2(x,t) = v_1^s(x,t)+v_2^s(x,t)$ for
  $(x,t)\in\Gamma\times \textcolor{r1}{\R}$, to establish the $M=2$
  result it suffices to show that
  \begin{equation}\label{relation}
    v^s_1(x,t)+v^s_2(x,t)+u^i(x,t)=0\quad\mbox{for}\quad (x,t)\in\Gamma\times
  (-\infty,\delta_{12}/c],
  \end{equation}
  which, in view of (\ref{criterion0}), results from the conditions
  \begin{eqnarray}
\label{condition21.2}
v^s_1(x,t)+v^s_2(x,t)+u^i(x,t)=0\quad\mbox{for}\quad
    (x,t)\in\Gamma_2\times \textcolor{r1}{\R}
  \end{eqnarray} and
  \begin{eqnarray}
\label{condition21.3}
v^s_2(x,t)=0\quad\mbox{for}\quad
(x,t)\in\Gamma_1^{\rm tr}\times(-\infty,\delta_{12}/c].
\end{eqnarray}
Equation (\ref{condition21.2}) follows immediately from Definition~\ref{inductive} since per~\eqref{bdry_m} and~\eqref{criterion0} we have $v^s_2(x,t) = f_2(x,t)= -u^i(x,t)-v^s_1(x,t)$ on $\Gamma_2$ for all $t\in\mathbb{R}$. To verify (\ref{condition21.3}), we note from (\ref{vanishing}) that $f_2(x,t)$ vanishes for $(x,t)\in \Gamma_{12}\times \textcolor{r1}{\R}$. Then in view of \textcolor{r1}{Lemma}~\ref{localpro}, equation (\ref{condition21.3}) results and thus, the proof for the case $M=2$ follows.

Using the notation $j'(m)=\mathrm{mod}(j(m),2)+1$
(equation~\eqref{prime}), to complete the inductive proof we assume
that for any $M\in \mathbb{N}$ with $2\le M\le L, L\geq 2$, the
following two relations hold: \begin{eqnarray}
\label{condition21.4}
u^s_M(x,t)+u^i(x,t)=0\quad\mbox{for}\quad (x,t)\in\Gamma_{j(M)}\times
  \textcolor{r1}{\R},
\end{eqnarray}
and
\begin{eqnarray}
\label{condition21.5}
v^s_{M}(x,t)=0\quad\mbox{for}\quad (x,t)\in\Gamma_{j'(M)}^{\rm
  tr}\times(-\infty,(M-1)\delta_{12}/c].
\end{eqnarray}
We then show that the
same relations and, as a result, the relation
(\ref{bound_0}), hold for $M=L+1$. To do this we note that equation
(\ref{criterion2}) tells us that $v^s_{L+1}(x,t)+v^s_{L}(x,t)=0$ for
$(x,t)\in\Gamma_{j(L+1)}\times\textcolor{r1}{\R}$. Therefore, the
$M=L-1$ condition (\ref{condition21.4})  implies that
\begin{eqnarray}
\label{condition21.6}
u^s_{L+1}(x,t)+u^i(x,t)=v^s_{L+1}(x,t)+v^s_{L}(x,t)+u^s_{L-1}(x,t)+u^i(x,t)=0
\end{eqnarray}
for $(x,t)\in\Gamma_{j(L+1)}\times\textcolor{r1}{\R}$. Noting that
$j'(L)=j(L+1)$ and using (\ref{vanishing}) and
(\ref{condition21.5}) with $M=L$ we see that $f_{L+1}(x,t)=0$ for
$(x,t)\in\Gamma_{j(L+1)}\times (-\infty,(L-1)\delta_{12}/c]\cup
\Gamma_{12}\times \textcolor{r1}{\R}$, and, thus, \textcolor{r1}{Lemma}~\ref{localpro}
tells us that \begin{equation} v^s_{L+1}(x,t)=0\quad\mbox{for}\quad
(x,t)\in\Gamma_{j'(L+1)}^{\rm tr}\times(-\infty,L\delta_{12}/c]  \end{equation}
---or, in other words, the relations (\ref{condition21.4}) and (\ref{condition21.5})
hold for $M = L+1$. Combining the relation (\ref{condition21.5}) and
the condition (\ref{condition21.4}) for $M=L$, it follows that \begin{eqnarray}
\label{condition21.7}
u^s_{L+1}(x,t)+u^i(x,t)=v^s_{L+1}(x,t)+u^s_{L}(x,t)+u^i(x,t)=0
\end{eqnarray}
for $(x,t)\in\Gamma_{j'(L+1)}^{\rm tr}\times(-\infty,L\delta_{12}/c]$. The relation (\ref{bound_0}) for $M=L+1$ results from (\ref{condition21.6}) and (\ref{condition21.7}), which completes the proof.
\end{proof}

\begin{remark}
  \textcolor{r1}{As detailed in Section~\ref{sec:3.2}, a variant of
    the setting considered in Theorem~\ref{equivalence}, involving
    certain ``extended'' open arcs $\widetilde\Gamma_j$, is utilized
    in the actual numerical implementation we propose. The use of
    extended arcs eliminates numerical accuracy losses that arise from
    the solution singularities that exist at the open-arc
    endpoints. As indicated in that section, the theorem and proof
    remain essentially unchanged.}
\end{remark}

The proposed multiple scattering strategy for the solution of the wave
equation problem (\ref{waveeqn}) for $(x,t)\in\Omega\times (-\infty,T(M)]$
($M=2,3,\dots$), which is embodied in Theorem~\ref{equivalence}, the
associated ping-pong solutions $v_m^s$, and the sum~\eqref{mult_scatt}
($m=1,\cdots,M$), is summarized in Algorithm~\ref{alg0}. Note that, in this algorithm, the necessary solutions $v^s_m(x,t)$ are obtained by means of the hybrid frequency-time approach presented in Section~\ref{sec:2.5}.

\begin{algorithm}[H]
\caption{Multiple scattering algorithm}
\label{alg0}
\begin{algorithmic}[1]
  \STATE Do $m=1,2,\cdots,M$ \STATE \ \ \ \ Evaluate the boundary data
  $f_m$ via relations (\ref{criterion0})-(\ref{criterion2}).  \STATE \
  \ \ \ Compute
  $v^s_{m}(x,t),
  (x,t)\in(\Omega\cup\Gamma_{j'(m)})\times(-\infty,T(M)]$ \textcolor{r1}{in
  Definition~\ref{inductive}} using the open-arc hybrid \\
  \ \ \ \ solver
  presented in Section~\ref{sec:2.5}.  \STATE End Do \STATE Compute
  $u^s_{M}(x,t), (x,t)\in\Omega\times(-\infty,T(M)]$ using
  equation~\eqref{mult_scatt}.
\end{algorithmic}
\end{algorithm}


\subsection{Frequency-domain multiple scattering algorithm}
\label{sec:2.5}

Call $F(\omega)$ the Fourier transform of a function
$f(t)\in L^2(\R)$,
\begin{eqnarray}
\label{forwardFT}
F(\omega)=\mathbb{F}(f)(\omega):=\int_{-\infty}^{+\infty}
f(t)e^{i\omega t}\,dt,  \end{eqnarray}
and let the corresponding inverse Fourier
transform of a frequency-domain function $F\in L^2(\R)$ be denoted by \begin{eqnarray}
\label{backwardFT}
f(t)=\mathbb{F}^{-1}(F)(\omega):=\frac{1}{2\pi}\int_{-\infty}^{+\infty}
F(\omega)e^{-i\omega t}\,d\omega.
\end{eqnarray}
Then, \textcolor{r1}{calling} $\kappa=\omega/c$ the spatial wave number, the Fourier
transform $V^s_{m}(x,\omega)$ of the solution $v^s_{m}(x,t)$ of the
wave equation is a solution of the Helmholtz equation
$\Delta V^s_{m} + \kappa^2 V^s_{m}=0$ in $\Omega_{j(m)}$ with
Dirichlet boundary conditions $V^s_{m}=F_{m}$ on $\Gamma_{j(m)}$ where
$F_{m}(x,\omega)=\mathbb{F}(f_m)(x,\omega)$. As is well known, the
solution $V^s_{m}(x,\omega)$ admits the
representation \begin{eqnarray}
\label{pingpong3}
V^s_{m}(x,\omega)=\mathcal{S}_{j(m)}[\psi_{m}](x,\omega):= \int_{\Gamma_{j(m)}} \Phi_\omega(x,y)\psi_{m}(y)ds_y,\quad x\in\Omega_{j(m)},
\end{eqnarray}
where $\psi_{m}$ is the solution of the integral equation
\begin{eqnarray}
\label{pingpong4}
S_{j(m)}[\psi_{m}]=F_{m}\quad\mbox{on}\quad \Gamma_{j(m)}.
\end{eqnarray}
Here\textcolor{r1}{, using the notations introduced in~\cite{SW84},}
$S_j: \widetilde{H}^{-1/2}(\Gamma_j)\rightarrow H^{1/2}(\Gamma_j)$
denotes the single-layer operator
\begin{equation}
\label{freq_sing}
S_j[\psi](x,\omega):= \int_{\Gamma_j} \Phi_\omega(x,y) \psi(y)ds_y, \quad x\in\Gamma_j,
\end{equation}
where\textcolor{r1}{, calling} $H_0^{(1)}$ the Hankel function of first kind \textcolor{r1}{and order
zero}, $\Phi_\omega(x,y)=\frac{i}{4}H_0^{(1)}(\kappa|x-y|)$ denotes the
fundamental solution \textcolor{r1}{associated with} the Helmholtz equation
$\Delta w+\kappa^2w=0$ in $\R^2$. Our approach relies on the existence
and uniqueness of solution of equation (\ref{pingpong4}), which are
guaranteed by the following theorem.

\begin{theorem}
\label{wellfrequency}
Given $F_{m}\in H^{1/2}(\Gamma_{j(m)})$ the integral equation (\ref{pingpong4}) admits a unique solution in $\widetilde{H}^{-1/2}(\Gamma_j)$ for any frequency $\omega>0$.
\begin{proof} Provided in~\cite{SW84}.\end{proof}
\end{theorem}

In view of Definition~\ref{inductive}, the boundary function $F_m$ is
determined inductively by the relations
\begin{equation}
  \label{criterion3}
  F_1(x,\omega)= -\mathbb{F}(u^i)(x,\omega)\quad\mbox{on}\quad\Gamma_1,\quad F_2(x,\omega)= -\mathbb{F}(u^i)(x,\omega)-V^s_{1,m-1}(x,\omega)\quad\mbox{on}\quad\Gamma_2,
\end{equation}
and,   \begin{equation}
  \label{criterion4}
  F_m(x,\omega)=-V^s_{m-1}(x,\omega)\quad\mbox{on}\quad\Gamma_{j(m)}, \quad m \geq 3.
\end{equation}
The frequency-domain component of the proposed frequency-time hybrid
multiple scattering algorithm, which produces the solutions $v_m^s$
for $m=1,\cdots,M$, is obtained by re-expressing Algorithm~\ref{alg0}
via an application of the Fourier transform. The result is
Algorithm~\ref{alg1} below.
\begin{algorithm}[H]
\caption{Hybrid multiple scattering algorithm}
\label{alg1}
\begin{algorithmic}[1]
\STATE Do $m=1,2,\cdots,M$
\STATE \ \ \ \ Evaluate the boundary data $F_m$ via relations (\ref{criterion3})-(\ref{criterion4}).
\STATE \ \ \ \ Solve the integral equation (\ref{pingpong4}) with solution $\psi_{m}$.
\STATE \ \ \ \ Compute $V^s_{m}(x,\omega), (x,\omega)\in(\Omega\cup\Gamma_{j'(m)})\times \textcolor{r1}{\R}$ via (\ref{pingpong3}).
\STATE \ \ \ \ Compute $v^s_{m}(x,t)=\mathbb{F}^{-1}(V^s_{m})(x,t), (x,t)\in(\Omega\cup\Gamma_{j'(m)})\times[0,T(M)]$.
\STATE End Do
\end{algorithmic}
\end{algorithm}

\subsection{Windowing and time-recentering}
\label{sec:2.6}

For a given signal $f_{m}(x,t)$, the time-domain open-arc solution
described in Section~\ref{sec:2.5} is obtained via the
following sequence of operations: \begin{equation}
\label{openarcprocess}
f_{m}(x,t)\xrightarrow{\mathbb{F}} F_{m}(x,\omega) \xrightarrow{(\ref{pingpong3}),(\ref{pingpong4})} V^s_{m}(x,\omega) \xrightarrow{\mathbb{F}^{-1}} v^s_{m}(x,t).
\end{equation}
Clearly, the function $f_{m}(x,t)$ may represent a signal of
arbitrarily long duration: this is merely a smooth compactly supported
function for $t\in[0,T]$, with a potentially large value of
$T>0$. For such large values of $T$ the Fourier transform
$F_{m}(x,\omega)$ is generally a highly oscillatory function of
$\omega$, as a result of the fast oscillations in the Fourier
transform integrand factor $e^{i\omega t}$---see
e.g. \cite[Fig. 1]{ABL18}. Under such a scenario a very fine
frequency-discretization, requiring $\mathcal{O}(T)$ frequency points,
and, thus, a number $\mathcal{O}(T)$ of evaluations of the
frequency-domain boundary integral equation solver, is required to
obtain the time-domain solution $v^s_{m}(x,t)$. This makes the overall
algorithm unacceptably expensive for long-time simulations. To
overcome these difficulties, a certain ``windowing and
time-recentering'' procedure was proposed in \cite[Sec. 3.1]{ABL18},
that decomposes a scattering problem involving an incident time signal
of long duration into a sequence of problems with smooth incident
field of a limited duration, all of which can be solved in terms of a
fixed set of solutions of the corresponding frequency-domain problems
for arbitrarily large values of $T$.

\begin{figure}[htb]
\centering
\includegraphics[scale=0.2]{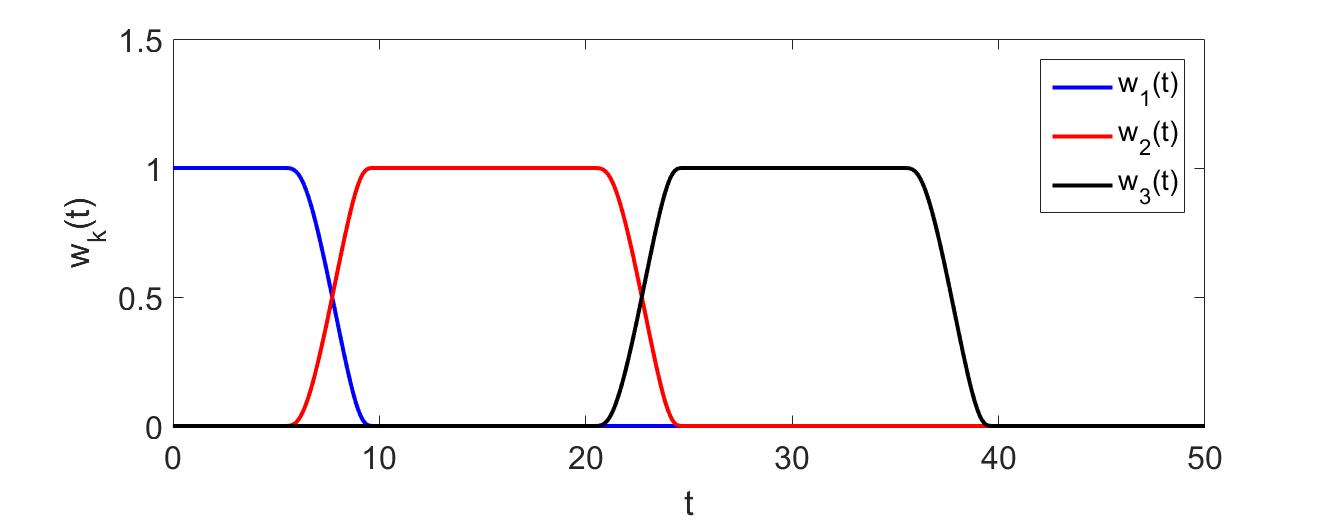}
\caption{Windowing functions $w_k(t), k=1,2,3$ with $H=10$.}
\label{timewindow}
\end{figure}

For a given final time $T$, the windowing-and-recentering approach is
based on use of a smooth partition of unity
$\mathcal{P}=\{\chi_k(t)\ |\ k\in\mathcal{K}\},
\mathcal{K}=\{1,\cdots,K\}$, where the functions $\chi_k$ satisfy $\sum_{k\in\mathcal{K}}\chi_k(t)\ = 1$ for $t\in[0,T]$ and \textcolor{r1}{where}, for a certain sequence $s_k$ ($k\in\mathcal{K}$), each $\chi_k(t)$ is a non-negative,
smooth windowing function of $t$, supported in the interval
$[s_k-H,s_k+H]$ of duration $2H$. The partition-of-unity $\mathcal{P}$ can be generated on
the basis of the smooth function $\eta(t;t_0,t_1)$ given by \begin{eqnarray}
\label{eta}
\eta(t;t_0,t_1)=\begin{cases}
1, & |t|\le t_0, \cr
e^{\frac{2e^{-1/s}}{s-1}}, & t_0<|t|<t_1, s=\frac{|t|-t_0}{t_1-t_0}, \cr
0, & |t|\ge t_1.
\end{cases}
\end{eqnarray}
Without loss of generality, in this work we set
\begin{equation}
H=\frac{T}{3K/2-1}, \quad s_k=\frac{3}{2}(k-1)H,
\end{equation}
and
\begin{equation*}
\chi_k(t)=\chi(t-s_k),\quad \chi(s)=\begin{cases}
\eta(s/H;1/2,1), & -H/2\le s\le H, \cr
1- \eta(s/H+3/2;1/2,1), & -H< s<-H/2, \cr
0, & |s|\ge H,
\end{cases}
\end{equation*}
---a prescription that  clearly ensures that $s_K+H/2=T$ and $\sum_{k=1}^K \chi_k(t)=1$ for all
$t\in[0,T]$. A depiction of such a partition of unity, with $H=10$, is
presented in Figure~\ref{timewindow}.

Utilizing the partition-of-unity $\mathcal{P}$, any smooth long-time
signal $f(t)$, $t\in[0,T]$, can be expressed in the form
\begin{equation}
  f(t)=\sum_{k\in\mathcal{K}} f_k(t), \quad f_k(t)=f(t)\chi_k(t),
\end{equation}
where $f_k$ is compactly supported in $[s_k-H,s_k+H]$. The
corresponding Fourier transform is then given by
\begin{equation}
F(\omega)=\sum_{k\in\mathcal{K}} F_k(\omega),\quad
F_k(\omega)=\int_{0}^{T_0} f_k(t)e^{i\omega t}dt= e^{i\omega
  s_k}F_{k,slow}(\omega),
\end{equation}
where, defining by
\ben
\mathbb{F}_{k,slow}(f)(\omega):=\int_{-H}^H
  f(t+s_k)\chi_k(t)e^{i\omega t}dt,
\enn
the $s_k$-centered slow Fourier-transform operator, we call
$F_{k,slow}(\omega)= \mathbb{F}_{k,slow}(f)(\omega)$; note that, as suggested by the notation used, $F_{k,slow}$ is a slowly-oscillatory function of $\omega$.

For $k\in\mathcal{K}$, we now call $F_{m,k}=\mathbb{F}_{k,slow}(f_{m})$ the slow Fourier-transform of the $m$-th iterate $f_m$, and we let
\begin{eqnarray}
\label{sol3}
V^s_{m,k}(x,\omega)=\mathcal{S}_{j(m)}[\psi_{m,k}](x,\omega),
\end{eqnarray}
where $\psi_{m,k}$ is the solution of the integral equation
\begin{eqnarray}
\label{BIE3}
S_{j(m)}[\psi_{m,k}]=F_{m,k}\quad\mbox{on}\quad \Gamma_{j(m)}.
\end{eqnarray}
It follows that
\begin{equation}
F_{m}(x,\omega)=\sum_{k\in\mathcal{K}} e^{i\omega s_k}F_{m,k}(x,\omega),
\end{equation}
and
\begin{eqnarray}
\label{seriessol2}
  v^s_{m}(x,t)=\sum_{k\in\mathcal{K}}
  \mathbb{F}^{-1}(V^s_{m,k})(x,t-s_k); \end{eqnarray} note that
$\mathbb{F}^{-1}(V^s_{m,k})(x,t)=0$ for $t<s_k-H$. Adopting the
time-recentering strategy described in this section,  Algorithm~\ref{alg1} leads to the more
efficient Algorithm~\ref{alg11}.

\begin{algorithm}
\caption{Hybrid multiple scattering algorithm with time-recentering}
\label{alg11}
\begin{algorithmic}[1]
\STATE Do $m=1,2,\cdots,M$
\STATE \ \ \ \ Evaluate the boundary data $f_m(x,t), (x,t)\in\Gamma_{j(m)}\times[0,s_K+H]$ via relations (\ref{criterion0})-(\ref{criterion2}).
\STATE \ \ \ \ Set $v^s_{m}(x,t)=0, (x,t)\in(\Omega\times[0,T])\cup(\Gamma_{j'(m)}\times[0,s_K+H])$.
\STATE \ \ \ \ Do $k=1,2,\cdots,K$
\STATE \ \ \ \ \ \ \ \ Evaluate the boundary data $F_{m,k}(x,\omega)=\mathbb{F}_{k,slow}(f_m)(x,\omega), (x,\omega)\in\Gamma_{j(m)}\times \textcolor{r1}{\R}$.
\STATE \ \ \ \ \ \ \ \ Solve the integral equation (\ref{BIE3}) with solution $\psi_{m,k}$.
\STATE \ \ \ \ \ \ \ \ Compute $V^s_{m,k}(x,\omega), (x,\omega)\in(\Omega\cup\Gamma_{j'(m)})\times \textcolor{r1}{\R}$ through (\ref{sol3}).
\STATE \ \ \ \ \ \ \ \ Compute $v^s_{m}(x,t)+=\mathbb{F}^{-1}(V^s_{m,k})(x,t-s_k), (x,t)\in(\Omega\times[0,T])\cup(\Gamma_{j'(m)}\times[0,s_K+H])$.
\STATE \ \ \ \ End Do
\STATE End Do
\end{algorithmic}
\end{algorithm}

\section{Hybrid multiple scattering strategy: numerical implementation}
\label{sec:3}

This section presents algorithms necessary for the numerical
implementation of the hybrid multiple scattering strategy introduced
in Algorithm~\ref{alg11}, including algorithms for accurate evaluation
of layer potentials, boundary integral operators, and  inverse Fourier transforms of certain singular functions.

\subsection{Fourier transform algorithm}
\label{sec:3.1}

Recalling the forward and inverse Fourier transform expressions (\ref{forwardFT}) and (\ref{backwardFT}),
we note that, for smooth and compactly supported functions $f$, the corresponding Fourier transforms $F$ decay superalgebraically fast (i.e., faster than any negative power of $\omega$) as
$\omega\rightarrow\pm\infty$. Thus, the errors in approximation \begin{eqnarray}
\label{backwardFT1}
f(t)\approx \frac{1}{2\pi}\int_{-W}^{W} F(\omega)e^{-i\omega t}d\omega \end{eqnarray}
decays super-algebraically fast as $W\rightarrow\infty$ \textcolor{r1}{in $H^s([0,T])$-norm, for any $s\ge 0$, as it follows easily by iterated integration by parts}:
the infinite-domain Fourier transform integral can be replaced by the corresponding integral over a finite interval with superalgebraically small errors.
As is known~\cite{M65,W86}, however, the frequency-domain solutions of
the Helmholtz equation in two dimensions vary as an integrable
function of $\log\omega$ which vanishes at $\omega=0$, and, thus, the
integration process requires some care to produce the needed integrals
with high-order accuracy. To do this, in what follows we employ the
recently developed Fourier-continuation (FC) based
approach~\cite{ABL18} for the numerical evaluation of such singular inverse
Fourier transform integrals.

Thus, utilizing a decomposition of the form \begin{eqnarray}
\label{backwardFT2}
f(t)=\frac{1}{2\pi}\left(\int_{-W}^{-w_c}+\int_{-w_c}^{w_c}+\int_{w_c}^{W}\right) F(\omega)e^{-i\omega t}d\omega,
\end{eqnarray}
we only need to consider 1)~Integrals of the form
\begin{eqnarray}
\label{int1}
I_a^b[F](t)=\int_{a}^{b} F(\omega)e^{-i\omega t}d\omega,
\end{eqnarray}
where $F$ is a smooth non-periodic function, and 2) The half-interval integrals
\begin{eqnarray}
\label{int2}
I_0^{w_c}[F](t)=\int_{0}^{w_c} F(\omega)e^{-i\omega t}d\omega\quad\mbox{and}\quad I_{-w_c}^0[F](t)=\int_{-w_c}^0 F(\omega)e^{-i\omega t}d\omega,
\end{eqnarray}
where $F(\omega)$ contains a logarithmic singularity at $\omega=0$.

To treat the integral $I_a^b[F](t)$, we re-express it in the form
\begin{eqnarray}
\label{int3}
I_a^b[F](t)=e^{-i\delta t}\int_{-A}^{A} F(\omega+\delta)e^{-i\omega
  t}d\omega,\quad \delta=\frac{a+b}{2},\quad A=\frac{b-a}{2}.  \end{eqnarray}
Although $F(\omega+\delta)$ is not a periodic function of $\omega$
in the integration interval $[-A,A]$,  it can be approximated, in this interval, by a
Fourier-continuation trigonometric polynomial~\cite{BL10}
\begin{eqnarray}
\label{fcpoly}
  F(\omega+\delta)=\sum_{m=-\widetilde{L}/2}^{\widetilde{L}/2-1} c_me^{i\frac{2\pi}{P}m\omega}
\end{eqnarray}
of a certain periodicity $P$, with high-order convergence as $\widetilde{L}$ grows. Indeed, an accurate
Fourier approximation of a certain period $P>2A$ can be obtained on
the basis of the FC(Gram) Fourier Continuation method~\cite{AB16,BL10}
from which the approximation errors decay as a user-prescribed
negative power of $\widetilde{L}$. Substituting (\ref{fcpoly}) into
(\ref{int3}) and integrating term-wise gives the approximation
\begin{eqnarray}
\label{int1result}
I_a^b[F](t) &=& e^{-i\delta t}\sum_{m=-\widetilde{L}/2}^{\widetilde{L}/2-1} c_m\int_{-A}^{A} e^{-i\frac{2\pi}{P}(\alpha t-m)\omega}d\omega \nonumber\\
&=& e^{-i\delta t}\sum_{m=-\widetilde{L}/2}^{\widetilde{L}/2-1} c_m \frac{P}{\pi(\alpha t-m)}\sin\left(\frac{2\pi A}{P}(\alpha t-m)\right),
\end{eqnarray}
with errors that are uniform in the time variable $t$. For a given user-prescribed equi-spaced
time-evaluation grid $\{t_n=n\Delta t\}_{n=N_1}^{N_2}$, the quantities
$I_a^b[F](t_n)$ can be obtained via an FFT-accelerated evaluation
of scaled discrete convolutions, see Section 4.1.2 in~\cite{ABL18} for
more details. But here, for simplicity, we evaluate the quantities
$I_a^b[F](t_n)$ directly.


In order to evaluate the integral $I_0^{w_c}[F](t)$ at fixed cost for arbitrarily \textcolor{r1}{large}
times $t$, in turn, we utilize a certain modified ``Filon-Clenshaw-Curtis'' high-order quadrature
approach  developed in~\cite{ABL18} which relies on a graded set
\ben
\left\{\mu_j=w_c\left(\frac{j}{Q}\right)^q, j=1,\cdots,Q\right\},
\enn
of points in the interval $(0,w_c)$ and associated integration subintervals $(\mu_j,\mu_{j+1}), j=1,\cdots,Q$. The integral $I_0^{w_c}[F](t)$ is thus approximated in accordance with the expression
\ben
I_0^{w_c}[F](t)=\sum_{j=1}^{Q-1} I_{\mu_j}^{\mu_{j+1}}[F](t),
\enn
in which the integral $I_{\mu_j}^{\mu_{j+1}}[F](t)$ is obtained via the  Clenshaw-Curtis
quadrature rule. This algorithm results in high-order convergence in spite of the logarithmic singular character of the function $F$. In detail, letting $n_{\mathrm{ch}}$ denote the selected number of Clenshaw-Curtis mesh points and assuming that $q>n_{\mathrm{ch}}+1$, the errors resulting from this approximation strategy decay as $\mathcal{O}(Q^{-(n_{\mathrm{ch}}+1)})$ as $n_{\mathrm{ch}}\rightarrow+\infty$.

Algorithm~\ref{alg11} also requires the evaluation of the Fourier transform
\ben
\mathbb{F}_{k,slow}(f)(\omega):=\int_{-H}^H
  f(t+s_k)\chi_k(t)e^{i\omega t}dt
\enn
for the smooth boundary-values function $f$. This computation proceeds in a manner analogous to that used for the evaluation of $I_{w_c}^W[F](t)$, except that, instead of Fourier continuation approximation of the function $F$ used in that case, here a regular Fourier expansion
\be
\label{fpoly2}
f(t+s_k)\chi_k(t)=\sum_{m=-\widehat{L}/2}^{\widehat{L}/2-1} \widehat{c}_m^ke^{i\frac{\pi}{H}mt},
\en
of periodicity interval $[-H,H]$, is used---which results in high-order convergence on account of the smooth vanishing of the function $\chi_k(t)$ at the endpoints of the interval $[-H,H]$. \textcolor{r1}{The approximation of $\mathbb{F}_{k,slow}(f)(\omega)$ is then obtained  via an expression analogous to (\ref{int1result})---with uniform errors  for all $\omega\in\mathbb{R}$, which are determined solely by the error in the approximation (\ref{fpoly2}).}

\subsection{Layer-potentials and integral-operator evaluations}
\label{sec:3.2}

The numerical implementation of the hybrid multiple scattering
strategy additionally requires  evaluation of the layer-potentials
$\mathcal{S}_j$, $j=1,2$  and the integral operators $S_j$, $j=1,2$ (equations (\ref{pingpong3})and (\ref{freq_sing}), respectively), both of which can be expressed as integrals of the form
\begin{eqnarray}
\label{singular2}
\mathcal{H}(x,\omega)&=&\int_{\Gamma_j} \Phi_\omega(x,y)\psi(y,\omega)ds_y, \quad x\in\Gamma_j,\;\;\textcolor{r1}{x\in\Omega}\;\;\mbox{or}\;\; \textcolor{r1}{x\in\Gamma\backslash\Gamma_j},
\end{eqnarray}
for certain densities $\psi(y,\omega)$. Depending on the location of
observation point $x$, the integral $\mathcal{H}(x,\omega)$ may be
weakly-singular, nearly-singular or non-singular. The numerical
evaluation of $\mathcal{H}(x,\omega)$ with high accuracy can be
achieved by means of a suitably modified version of the
two-dimensional Chebyshev-based rectangular-polar discretization
method~\cite{BY21} (cf.~\cite{BG18}) which adequately accounts for the
singular character of the unknown potential $\psi$ at the endpoints of
the open-arcs $\Gamma_j$. In detail~\cite{CDD03}, the density function
$\psi$ can be expressed in the forms $\psi=\alpha/w$ near the
endpoints where $\alpha$ is a smooth function and $w\sim d_j^{1/2}$
\textcolor{r1}{where} $d_j$ denotes the distance to the endpoint of $\Gamma_j$. Then a
special change of variables introduced in \cite[Eq. (4.12)]{BY20} (see
also~\cite{At91,BL12}), \textcolor{r1}{which eliminates the $1/w$ singularity, is utilized here} to evaluate the integrals
$\mathcal{H}(x,\omega)$ with \textcolor{r1}{high-order accuracy}.

\begin{figure}[htbp]
\centering
\includegraphics[scale=0.45]{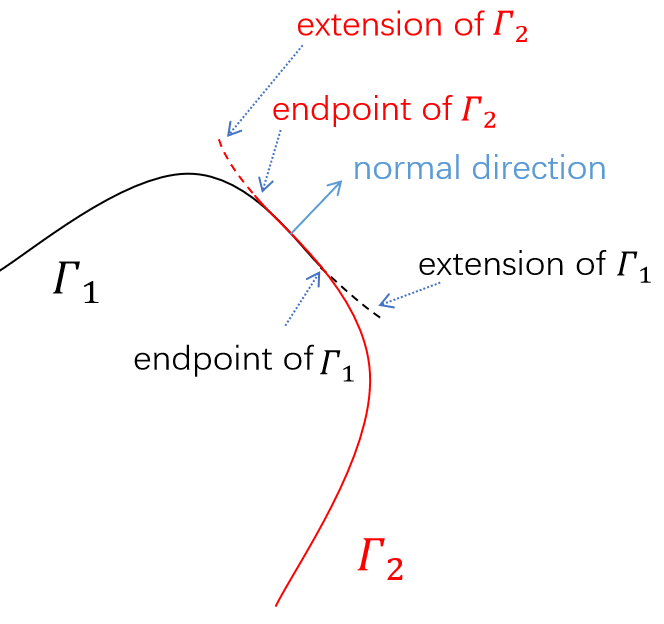}
\caption{Extended arcs utilized in the numerical implementation.}
\label{extension}
\end{figure}

In view of the aforementioned density singularity, it can be easily
shown that the collected boundary data $F_{m}$ ($m\ge 1$) is also
singular: it behaves like $d_{j(m)}^{1/2}$ near the endpoints of the
arc $\Gamma_{j(m)}$. Graded meshes near the endpoints could be
employed to ensure high-order accuracy in the solution of the
associated boundary integral equations. But a different approach is
utilized in this paper, whereby the edge singularity in the boundary
data may be entirely avoided by slightly and smoothly extending the
boundary $\Gamma_j$ in the direction normal to $\Gamma_j$---as
illustrated in Figure~\ref{extension}. More precisely, letting
$x=x(s)$ and $\nu=\nu(s)$ denote a parametrization of
\textcolor{r1}{$\Gamma =\Gamma_1\cup \Gamma_2$} and its normal vector, respectively,
the curve $\Gamma_j$ is prolonged \textcolor{r1}{beyond its endpoints} into an
extended open arc $\widetilde{\Gamma}_j$. \textcolor{r1}{Here} the extension arc,
denoted by $\Gamma_j^e=\widetilde{\Gamma}_j\backslash\Gamma_j$, is
given by
\begin{equation} y(s)=x(s)+a(s)\cdot\nu(s)
\end{equation}
\textcolor{r1}{for $s$ in a neighborhood beyond each parameter value $s=s_0$
corresponding to an endpoint of $\Gamma_j$. In order to ensure
sufficient smoothness, leading to high-order accuracy, a certain
number of derivatives of the function $a$ are required to vanish at
$s=s_0$.} \textcolor{r1}{The extended arcs $\widetilde{\Gamma}_j$ used are such that their endpoints are far from the region where the corresponding fields must be evaluated. Since the problems for the extended open-arcs are handled with high accuracy (by means of the numerical method~\cite{BG18,BY21}), and since the corresponding solutions are evaluated away from the extended-arc singular points, the difficulties arising from endpoint singularities are completely eliminated.}
\textcolor{r1}{\begin{remark}
\label{incext}
The incident field $u^i$ can easily be extended to each one of the two curves $\widetilde{\Gamma}_j$, $j=1,2$; the corresponding extensions will be denoted in what follows by $u^i_j$, $j=1,2$. The necessary extensions can be obtained either by simply evaluating on the extended curves an incident field function $u^i$ defined in all of $\R^2$, whenever, as is often the case, such a function is provided, or, alternatively, by using a Sobolev extension theorem such as \cite[Theorem 3.10]{Ne12} on each curve $\widetilde{\Gamma}_j$.
\end{remark}}

\textcolor{r1}{This extension procedure, which provides great flexibility, does not
negatively affect any aspect of the proposed multiple scattering
algorithm. Indeed, letting $\widetilde{\Omega}_j=\R^2\backslash\widetilde{\Gamma}_j$ and considering the wave equation problem
\begin{eqnarray}
\label{eq_op11}
\begin{cases}
\frac{\pa^2\widetilde{w}^s_{j}}{\pa t^2}(x,t)-c^2\Delta
\widetilde{w}^s_{j}(x,t)=0, &(x,t)\in\widetilde{\Omega}_j\times \R, \cr
\widetilde{w}_j^s(x,t)=\widetilde{g}_j(x,t),  & (x,t)\in\widetilde{\Gamma}_j\times\R,
\end{cases}
\end{eqnarray}
we have the following result analogous to Lemma~\ref{localpro}.
\begin{lemma}
\label{localpro11}
Let $j\in\{1,2\}$, $p\in\R$, and $T_0>0$, and assume that, for
$j\in\{1,2\}$, (a)~$\widetilde{g}_j\in H_{\sigma,T_0}^{p}(\R,H^{1/2}(\widetilde{\Gamma}_j))$
satisfies
\begin{equation}\label{ini_set11}
  \widetilde{g}_j(x,t)=0\quad\mbox{for}\quad
  (x,t)\in(\Gamma_{12}\cup\Gamma_j^e)\times \R;
\end{equation}
and, (b)~$\widetilde{\Gamma}_j$ satisfies Condition~\ref{localpro0}. Then,
recalling equation~\eqref{dist}, letting $t_0=\delta_{12}/c>0$, and
calling
$\widetilde{w}_j^s\in H_{\sigma,T_0}^{p-3}(\R,H^1(\widetilde{\Omega}_j))$ the
  unique solution of the wave equation problem \eqref{eq_op11}, we have
\begin{equation}\label{ini_set_res11}
  \widetilde{w}_j^s(x,t)=0\quad\mbox{for}\quad
  (x,t)\in \Gamma_{j'}^{\rm tr} \times (-\infty,T_0+t_0].
\end{equation}
\end{lemma}}

\textcolor{r1}{Definition~\ref{inductive}, in turn, needs to be adjusted as follows.
\begin{definition}\label{inductive11}
  For $m\in\mathbb{N}$ we inductively define $\widetilde{v}^s_m(x,t)$ as equal to
  the solution $\widetilde{w}^s_{j(m)}(x,t)$ of the open-arc
  problem~\eqref{eq_op11} with boundary data
  \begin{equation}\label{bdry_m11}
    \widetilde{v}^s_m(x,t) = \widetilde{f}_m(x,t)\quad \mbox{for}\quad (x,t)\in\widetilde{\Gamma}_{j(m)}\times \R,\quad (m \in\mathbb{N}),
  \end{equation}
  where $\widetilde{f}_{m}(x,t):\widetilde{\Gamma}_{j(m)}\times \R\to\mathbb{C}$ denotes the causal
  functions defined inductively via the relations
    \begin{equation}
  \label{criterion011}
  \widetilde{f}_1(x,t)= -u_1^i(x,t)\quad\mbox{on}\quad\widetilde\Gamma_1,
\end{equation}
\begin{equation}
  \label{criterion111}
  \widetilde{f}_2(x,t)=\begin{cases} -u_2^i(x,t)-\widetilde{v}^s_1(x,t), &  x\in\Gamma_2,\cr
  0,  & x\in\Gamma_2^e,
  \end{cases}
\end{equation}
and,   \begin{equation}
  \label{criterion211}
  \widetilde{f}_m(x,t)=\begin{cases}
  -\widetilde{v}^s_{m-1}(x,t), &  x\in\Gamma_{j(m)},\cr
  0,  & x\in\Gamma_{j(m)}^e,
  \end{cases} \quad m \geq 3.
\end{equation}
\end{definition}}

\textcolor{r1}{The new inductive relations give rise to the following slightly modified version of Theorem~\ref{equivalence}.
\begin{theorem}
\label{equivalence11}
Let $M\in\mathbb{N}$, $M\geq 2$, $p\in\R$, and let
  $T = T(M) = (M-1) \delta_{12}/c$. Denote an $M$-th order multiple-scattering sum
\begin{equation}\label{mult_scatt11}
  \widetilde{u}^s_M(x,t):=\sum_{m=1}^M
  \widetilde{v}^s_{m}(x,t),
\end{equation}
which includes contributions from all $M$ ``ping-pong'' scattering
iterates $\widetilde{v}^s_m(x,t)$ with $m=1,\dots, M$. Then, given $u^i\in H_{\sigma,0}^{p}(\R,H^{1/2}(\Gamma))$ and
  $u_j^i\in H_{\sigma,0}^{p}(\R,H^{1/2}(\widetilde{\Gamma}_j))$, $j=1,2$ as indicated in Remark~\ref{incext}, we have
  $\widetilde{u}^s_M\in H_{0}^{p-3/2}(T(M),H^1(\Omega))$ and
  $u^s(x,t) =\widetilde{u}^s_M(x,t)$ for all
  $(x,t)\in\Omega\times (-\infty,T(M)]$.
\end{theorem}}

\textcolor{r1}{The proof of this theorem is
  essentially identical to the proof of Theorem 2.8, and it is
  therefore omitted for brevity.}

  Incorporating the Fourier transform
and time-windowing and recentering strategies introduced in previous
sections, we are led to a new version of the hybrid multiple
scattering algorithm which, except for straightforward modifications
related to arc extensions, is entirely analogous to
Algorithm~\ref{alg11}, and whose slightly modified pseudocode is once
again omitted. The overall algorithm for evaluation of the numerical
solution $u^s$ of equation~\eqref{waveeqn}, incorporating the extended
arcs $\widetilde{\Gamma}_j$, is presented as Algorithm~\ref{alg2} in
Section~\ref{sec:3.3}.

Clearly, the weakly-singular integrals $\mathcal{H}(x,\omega)$ need to
be evaluated at a sufficiently large number of frequency
discretization-points $\omega$ in the interval $[-W,W]$. The
computational cost required for such evaluations can be reduced by
utilizing the decomposition
\begin{equation}
\Phi_\omega(x,y)=\Psi_0(x,y)+\kappa^2\Psi_1(x,y)+H_\omega(x,y),
\end{equation}
where
\begin{eqnarray}
\Psi_0(x,y)&=& -\frac{1}{2\pi}\log|x-y|,\\
\Psi_1(x,y)&=& \frac{|x-y|^2}{8\pi}\log|x-y|,
\end{eqnarray}
and
\begin{equation}
H_\omega(x,y)=\begin{cases}
\Phi_\omega(x,y)-\Phi_0(x,y)-\kappa^2\Phi_1(x,y), & x\ne y,\cr \\
\frac{i}{4}-\frac{1}{2\pi}(c_e+\log(\kappa/2)), & x=y,
\end{cases}
\end{equation}
($c_e=0.57721566\cdots$ is the Euler constant). The function
$H_\omega$ is more regular than the Green function $\Phi_\omega(x,y)$
itself, and its integration under a given error tolerance is therefore
less onerous. The discretization matrices associated with the
weakly-singular and nearly-singular integrals of the form
\begin{eqnarray}
\label{singular20}
  \mathcal{H}_\ell(x,\omega)
  &=&\int_{\widetilde\Gamma_j}
      \Psi_\ell(x,y)\psi(y,\omega)ds_y, \quad
      x\in\widetilde\Gamma_j,\;\;\Omega\;\;\mbox{or}\;\;
      \Gamma\backslash(\widetilde\Gamma_j),\quad \ell=0,1,
\end{eqnarray}
in turn, are independent of frequency, and can thus be precomputed
before the ping-pong iterative process is initiated.

In this work, the two-dimensional Chebyshev-based rectangular-polar
integral solver~\cite{BG18,BY21} is employed for the evaluation of all
singular integrals. The remaining integrals involving the smoother
kernels $H_\omega$ can be integrated efficiently and accurately by
means of Fejer's quadrature rule, and they can be further accelerated
e.g. by the methods presented in~\cite{BB21,BK01,L09} and references
therein, but such accelerations were not utilized in this work. In our numerical implementation, prior to the ping-pong iteration process we additionally use the Lapack function ZGESV to pre-compute the inverses $\mathbb{A}_j^{-1}(\omega)$ of the coefficient matrixes $\mathbb{A}_j(\omega)$ ($j=1,2$) resulting from the discretizations of the integral operators
\be
\label{singular3}
\mathcal{H}_1(x,\omega)+\textcolor{r1}{\kappa^2}\mathcal{H}_2(x,\omega)+\int_{\widetilde\Gamma_j} H_\omega(x,y)\psi(y,\omega)ds_y, \quad x\in\widetilde\Gamma_j,\quad j=1,2;
\en
these inverse matrices are then used repeatedly to obtain the numerical solution $\textcolor{r1}{\widetilde{\psi}_{m,k}}(x,\omega)$
of the integral equation
\be
\label{BIEmod}
\textcolor{r1}{\int_{\widetilde{\Gamma}_{j(m)}} \Phi_\omega(x,y)\widetilde{\psi}_{m,k}(y,\omega)ds_y=\widetilde{F}_{m,k}(x,\omega)\quad\mbox{on}\quad\widetilde{\Gamma}_{j(m)}}
\en
with \textcolor{r1}{$\widetilde{F}_{m,k}=\mathbb{F}_{k,slow}(\widetilde{f}_{m})$ and} $j=j(m)$ for all $k\in\mathcal{K}$ and all $m=1,\cdots,M$.

\subsection{Numerical implementation: overall outline}
\label{sec:3.3}

The overall algorithm for evaluation of the numerical solution $u^s$
of equation~\eqref{waveeqn} relies on the concepts presented in Sections \ref{sec:3.1}-\ref{sec:3.2} and the following notations and conventions.

With reference to Section~\ref{sec:3.1}, we denote by
$\mathcal{F}=\{\omega_1,\cdots,\omega_J\}$ a set of frequencies used
for the Fourier transformation process, which includes an equi-spaced
grid in the frequency intervals $[-W,-w_c]$ and $[w_c,W]$, as well as
a combination of the Clenshaw-Curtis mesh points in the intervals
$(-\mu_{j+1},-\mu_j)$ and $(\mu_j,\mu_{j+1})$, $j=1,\cdots,Q-1$,
for a total of $J=2\widetilde{L}+2n_{\mathrm{ch}}(Q-1)$ frequency
discretization points. For the necessary time-domain discretization, in turn, we use the mesh $\mathcal{T}=\{t_n=n\Delta t\}_{n=1}^{N_T}$
of the time interval $[0,s_K+H]$, where $\Delta
t=(s_K+H)/N_T$, and we call $\mathcal{T}_0=\mathcal{T}\cap[0,T]$.
With reference to Section~\ref{sec:3.2}, on the other hand, frequency-independent meshes $\mathcal{M}_{j}$ are used on the curves $\widetilde{\Gamma}_j$, $j=1,2$ for all frequencies considered. The set of discrete spatial observation points at which the scattered field is to be produced, finally,  is denoted by $\mathcal{R}$.

Using these notations,  a version of Algorithm~\ref{alg11}, including certain details concerning our numerical implementation, is presented in Algorithm~\ref{alg2}.

\begin{algorithm}
\caption{Numerical hybrid multiple scattering algorithm with time-recentering}
\label{alg2}
\begin{algorithmic}[1]
\STATE Pre-compute the matrices $\mathbb{A}_j^{-1}(\omega)$ for $j=1,2,\ \omega\in\mathcal{F}$.
\STATE Do $m=1,2,\cdots,M$
\STATE \ \ \ \ Evaluate the boundary data \textcolor{r1}{$\widetilde{f}_m(x,t), (x,t)\in\mathcal{M}_{j(m)}\times\mathcal{T}$ via relations (\ref{criterion011})-(\ref{criterion211})}.
\STATE \ \ \ \ Initialize $\textcolor{r1}{\widetilde{v}^s_{m}}(x,t)=0, (x,t)\in(\mathcal{R}\times\mathcal{T}_0)\cup(\mathcal{M}_{j(m)}\times\mathcal{T})$.
\STATE \ \ \ \ Do $k=1,2,\cdots,K$
\STATE \ \ \ \ \ \ \ \ For $\omega\in \mathcal{F}$, evaluate the vectors $\mathbb{B}_{m,k}(\omega)$ whose elements are \ben
\textcolor{r1}{\widetilde{F}_{m,k}(x,\omega)=\mathbb{F}_{k,slow}(\widetilde{f}_m)(x,\omega)}, \quad x\in\mathcal{M}_{j(m)},
\enn
\ \ \ \ \ \ \ \ by the Fourier transform algorithm presented in Section~\ref{sec:3.1}.
\STATE \ \ \ \ \ \ \ \ Compute the approximation of the solution $\textcolor{r1}{\widetilde{\psi}_{m,k}}(x,\omega), (x,\omega)\in \mathcal{M}_{j(m)}\times \mathcal{F}$ of the integral equation\\
\ \ \ \ \ \ \ \  \textcolor{r1}{(\ref{BIEmod})} given by $\mathbb{A}_j^{-1}(\omega)\mathbb{B}_{m,k}(\omega)$.
\STATE \ \ \ \ \ \ \ \ Evaluate $\textcolor{r1}{\widetilde{V}^s_{m,k}}(x,\omega), (x,\omega)\in(\mathcal{R}\cup\mathcal{M}_{j(m)})\times \mathcal{F}$ through
\ben
\textcolor{r1}{\widetilde{V}^s_{m,k}(x,\omega)=\int_{\widetilde{\Gamma}_{j(m)}} \Phi_\omega(x,y)\widetilde{\psi}_{m,k}(y,\omega)ds_y}
\enn
\ \ \ \ \ \ \ \ and rectangular-polar Chebyshev-based  integration (Section~\ref{sec:3.2}).
\STATE \ \ \ \ \ \ \ \ Evaluate $\textcolor{r1}{\widetilde{v}^s_{m}(x,t)=\widetilde{v}^s_{m}(x,t)+\mathbb{F}^{-1}(\widetilde{V}^s_{m,k})(x,t-s_k)}, (x,t)\in(\mathcal{R}\times\mathcal{T}_0)\cup(\mathcal{M}_{j(m)}\times\mathcal{T})$ by the Fourier \\
\ \ \ \ \ \ \ \ transform algorithm presented in Section~\ref{sec:3.1}.
\STATE \ \ \ \ End Do
\STATE End Do
\STATE Evaluate the numerical solution
      \begin{equation}
      u^s_{\rm num}(x,t)=\sum_{m=1}^M \textcolor{r1}{\widetilde{v}^s_{m}}(x,t),\quad (x,t)\in\mathcal{R}\times \mathcal{T}_0.
      \end{equation}
\end{algorithmic}
\end{algorithm}

\section{Numerical examples}
\label{sec:4}

This section presents a variety of numerical tests that illustrate the
character of the proposed frequency-time hybrid ping-pong
integral-equation solver embodied in Algorithm~\ref{alg2}. The numerical errors presented in this
section were calculated in accordance with the expression
$\max_{t\in[0,T]}|u^{s}_{\rm num}-u^{s}_{\rm ref}|$ where, with
exception of the test \textcolor{r1}{cases considered in Example 3 and 6, for which the
exact solutions are known, the reference solutions $u^{s}_{\rm ref}$ were
obtained as numerical solutions produced by means of} sufficiently fine
discretizations. (Our use of absolute errors is justified since, as evident from the numerical solution plots in each case,  we only consider solutions whose maximum values are quantities of order one.) All of the numerical tests were obtained on the basis of
Fortran numerical implementations, parallelized using OpenMP, on an
10-core HP Desktop with an Intel Core processor i9-10900.

{\bf Example 1.} Our first test case concerns the accuracy of the
numerical solver for the frequency domain integral equation
(\ref{pingpong4}) on the single open-arc $\Gamma_1$ shown in
Figure~\ref{MSmodel}(c), with point-source boundary data
\begin{equation}
  F_1=-U^i = - \frac{i}{4}H_0^{(1)}(\kappa|x|) \quad\mbox{on}\quad\Gamma_1.
\end{equation}
Figure~\ref{Example1.1} displays errors in the solution evaluated by
means of the single-layer potential (\ref{pingpong3}), at the points
$(0.5,0), (0,2), (0,1.01), (-0.99,0)$, as functions the number $N$ of
Chebyshev points used in each one of the patches associated with the
Chebyshev-based discretization methodology~\cite{BY21}. Clearly,
uniform fast convergence of the numerical solutions is obtained at all
points, independently of the distance to the boundary. For this
example a total of $5$ patches (resp. $25$ patches) where used for
test cases with wavenumber $\kappa=10$ (resp. $\kappa=50$).

\begin{figure}[htbp]
\centering
\begin{tabular}{cc}
\includegraphics[scale=0.18]{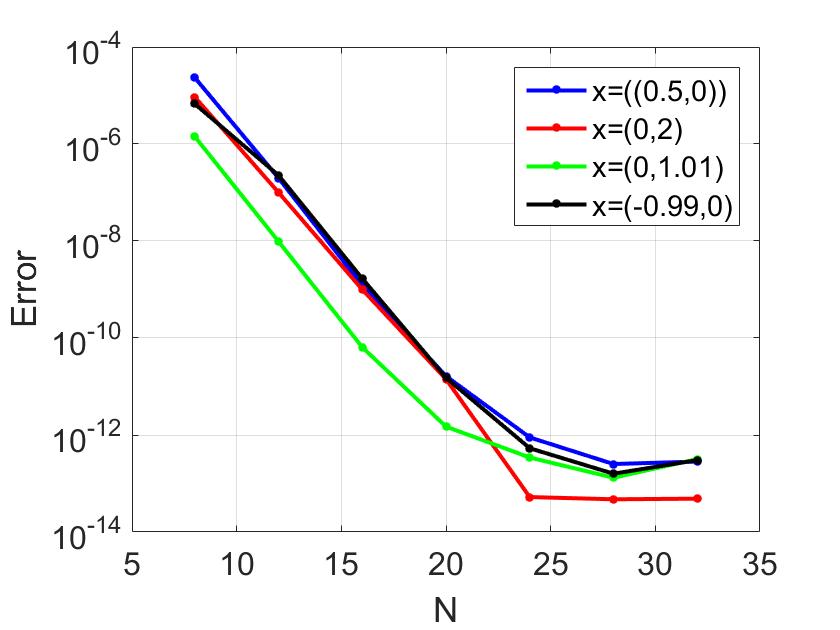} &
\includegraphics[scale=0.18]{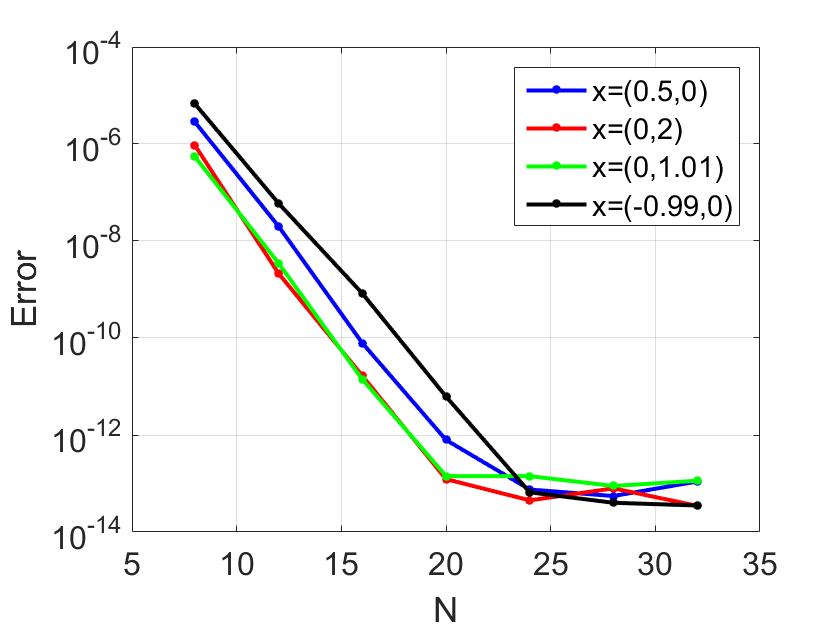} \\
(a) $\kappa=10$ & (b) $\kappa=50$
\end{tabular}
\caption{Numerical errors observed in the frequency-domain solutions considered in Example~1, at various points $x$, as functions of the number $N$ of discretization points used.}
\label{Example1.1}
\end{figure}

{\bf Example 2.} We now consider test cases that \textcolor{r1}{demonstrate} the
accuracy of the {\em time-domain} solver for the problem scattering by
a single open-arc $\Gamma_1$ depicted in Figure~\ref{MSmodel}(c). We
consider incident fields of two different kinds, namely, 1)~A
Gaussian-modulated point source $u_1^i(x,t)$ equal to the Fourier
transform of the function
\begin{eqnarray}
\label{pointsource}
U^i(x,\omega)=\frac{5i}{2}H_0^{(1)}(\omega|x-z|) e^{-\frac{(\omega-\omega_0)^2}{\sigma^2}}e^{i\omega t_0}
\end{eqnarray}
with respect to $\omega$, with $\sigma=2$, $\omega_0=15$, $t_0=4$ and
$z=(0,0)$; and 2)~A plane-wave incident field
\begin{eqnarray}
\label{planewave}
  u_2^i(x,t)=-\sin(4g(x,t))e^{-1.6(g(x,t)-3)^2},\quad
  g(x,t)=t-t_\textrm{lag}-x\cdot d^\textrm{inc}
\end{eqnarray}
along the incident direction $d^\textrm{inc}=(1,0)$  with
$t_\textrm{lag}=2$.  Together with a sufficiently fine fixed spatial
discretization, the fixed numerical frequency intervals
$\omega\in[5,25]$ and $\omega\in[-20,20]$ were used for the incident
fields $u_1^i$ and $u_2^i$, respectively. Figures~\ref{Example2.1} and
\ref{Example2.2} \textcolor{r1}{present the scattered field as a function of time $t$} at the observation point $x=(0.5,0)$ and the corresponding numerical errors at that point\textcolor{r1}{, respectively, as functions} of the number of frequencies used---demonstrating the
fast convergence of the algorithms as the frequency-domain
discretization is refined.

\begin{figure}[htbp]
\centering
\begin{tabular}{cc}
\includegraphics[scale=0.18]{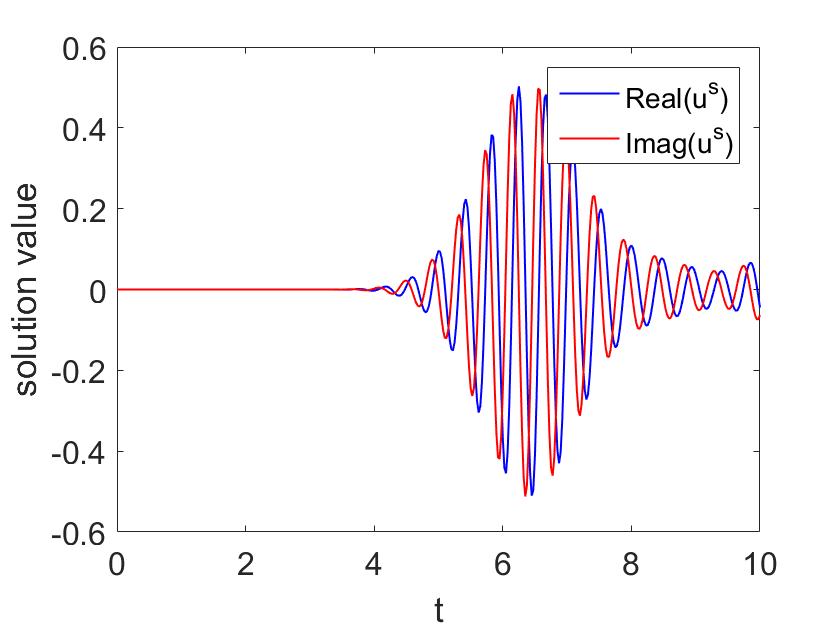} &
\includegraphics[scale=0.18]{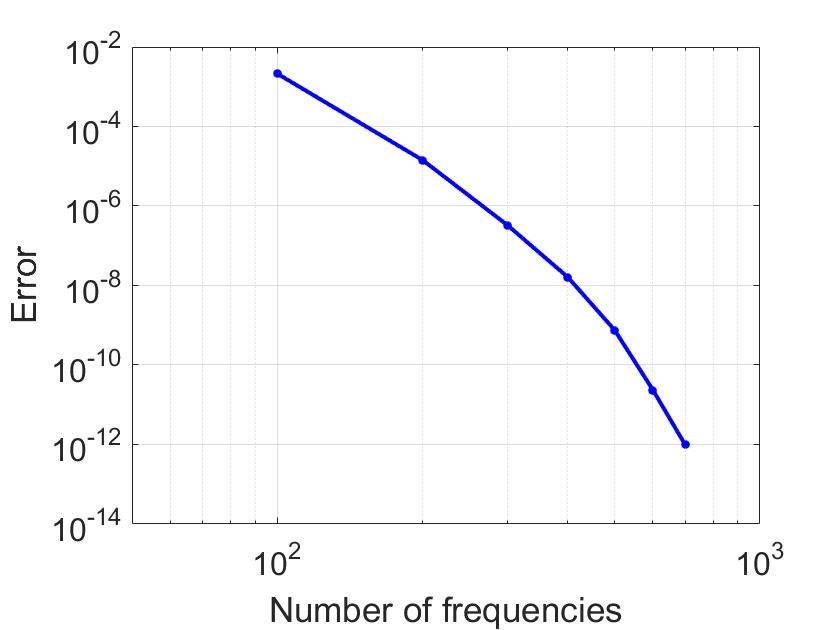} \\
(a) & (b)
\end{tabular}
\caption{Scattered field and errors obtained for the problem considered in Example 2. (a) Real and imaginary parts of scattered field  at $x=(0.5,0)$ resulting from the incident field $u_1^i$. (b) Convergence of the complex scattered field at $x=(0.5,0)$ as a function of the number of frequencies used.}
\label{Example2.1}
\end{figure}

\begin{figure}[htbp]
\centering
\begin{tabular}{cc}
\includegraphics[scale=0.18]{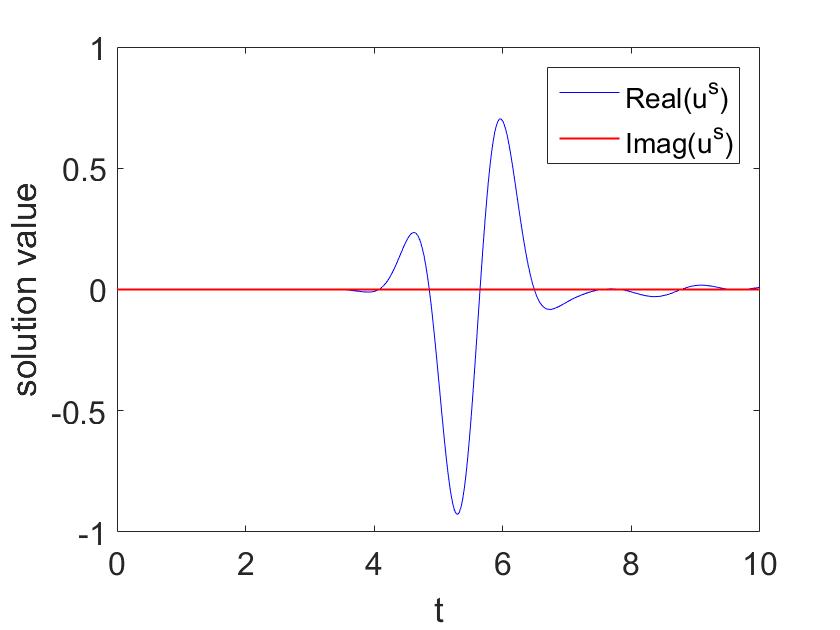} &
\includegraphics[scale=0.18]{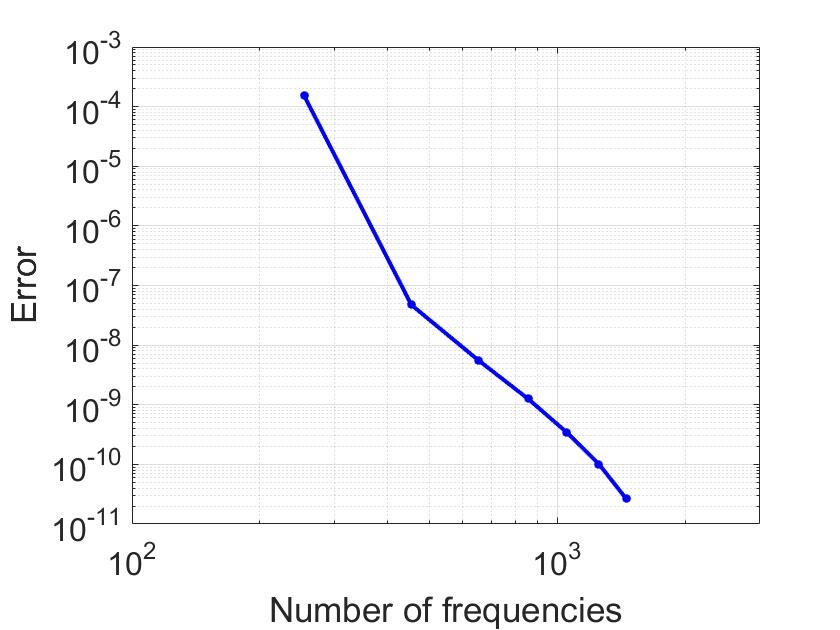} \\
(a) & (b)
\end{tabular}
\caption{Scattered field and errors obtained for the problem considered in Example 2. (a) Real and imaginary parts of scattered field  at $x=(0.5,0)$ resulting from the incident field $u_2^i$. (b) Convergence of the complex scattered field at $x=(0.5,0)$ as a function of the number of frequencies used.}
\label{Example2.2}
\end{figure}

{\bf Example 3.} This example presents the solutions produced by
the full hybrid ping-pong multiple scattering algorithm for the wave
equation problem (\ref{waveeqn}) in two different domains $\Omega$,
namely, the unit disc centered at the origin and the unit square
$\Omega=[-1,1]^2$, for $t\in[0,10]$, and for each one of the two
time-domain sources considered in Example 2: the point source
$u_1^i(x,t)$ and the plane wave source $u_2^i(x,t)$. For the
plane-wave incidence case the exact solution is given by
$u^s(x,t)=-u_2^i(x,t)$ for $x\in\Omega$. In this example, the
extensions $\widetilde{\Gamma}_j$ of $\Gamma_j$ for $j=1,2$ are
constructed by means of portions of tangent circular arcs of radii
$0.1$. For the wave equation problem (\ref{waveeqn}) in a unit disc domain, the numerical errors as a function of $M$ are displayed in Figures~\ref{Example3.1} and \ref{Example3.2}: clearly,
rapid convergence and high accuracy are
observed. Figures~\ref{Example3.3} and \ref{Example3.4} display the total
field within the rectangular domain $\Omega$ at various times, for two different point-source
locations $z$, and two different values of $\omega_0$ in~(\ref{pointsource}), and using a total of $M=10$ ping-pong iterations; we have verified that, in this case, the numerical errors are less than $10^{-8}$ for all $t\in[0,10]$.


\begin{figure}[htbp]
\centering
\begin{tabular}{cc}
\includegraphics[scale=0.2]{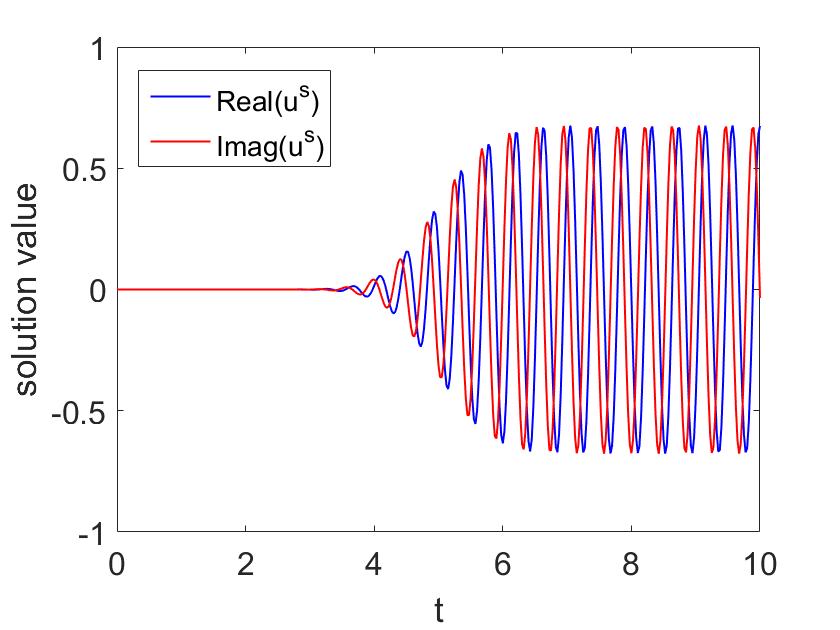} &
\includegraphics[scale=0.2]{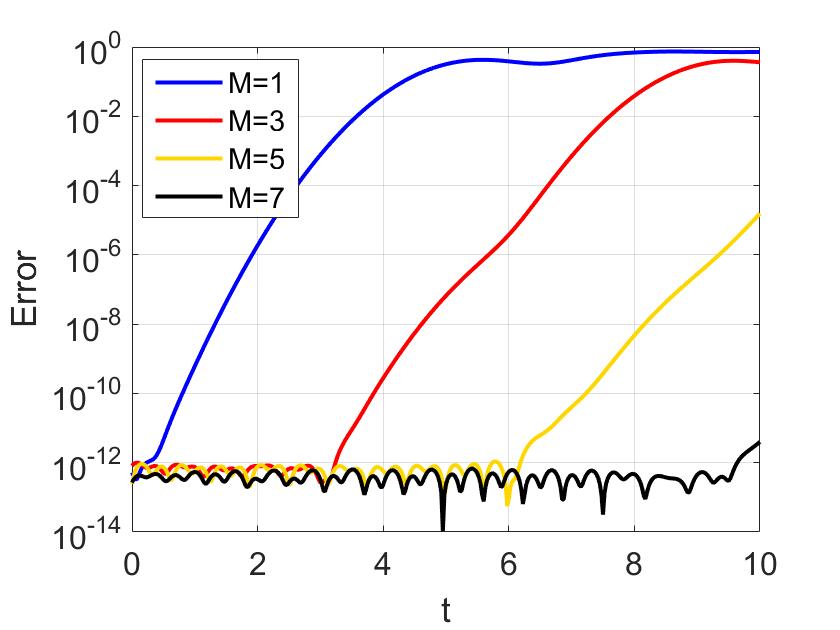} \\
(a) & (b)
\end{tabular}
\caption{Scattered field and errors obtained for the problem considered in Example 3. (a) Real and imaginary parts of scattered field  at $x=(0.5,0)$ resulting from the incident field $u_1^i$. (b) Numerical errors as functions of time $t$ for various values of $M$.}
\label{Example3.1}
\end{figure}

\begin{figure}[htbp]
\centering
\begin{tabular}{cc}
\includegraphics[scale=0.2]{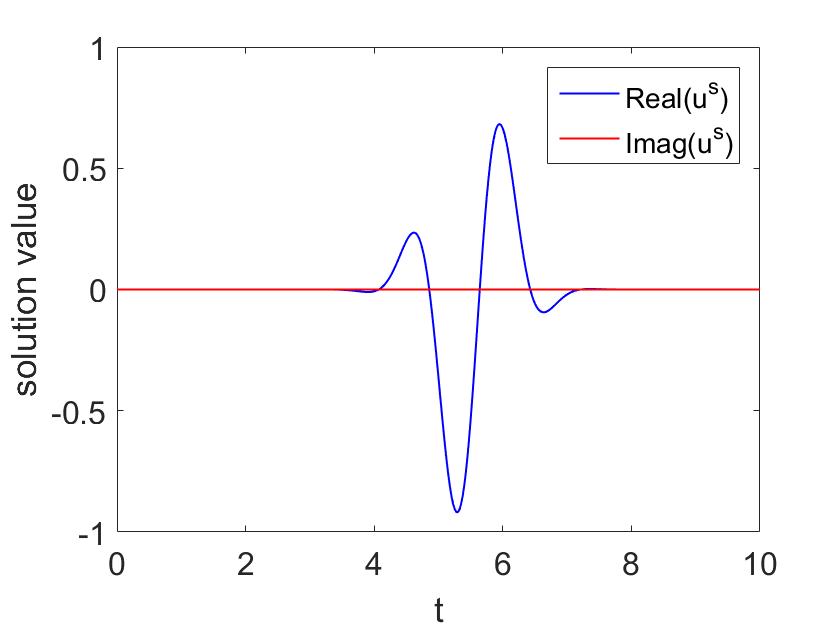} &
\includegraphics[scale=0.2]{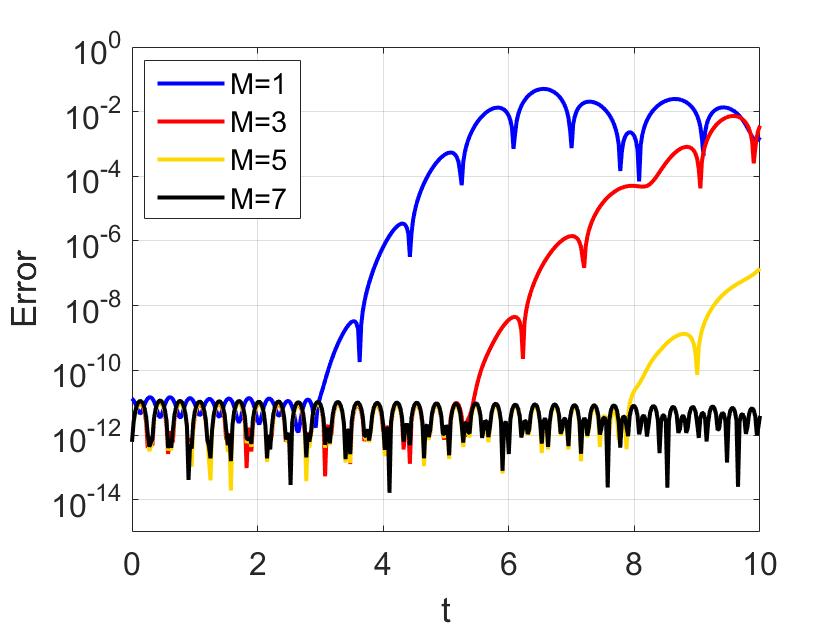} \\
(a) & (b)
\end{tabular}
\caption{Scattered field and errors obtained for the problem considered in Example 3. (a) Real and imaginary parts of scattered field  at $x=(0.5,0)$ resulting from the incident field $u_2^i$. (b) Numerical errors as functions of time $t$ for various values of $M$.}
\label{Example3.2}
\end{figure}

\begin{figure}[htbp]
\centering
\begin{tabular}{cc}
\includegraphics[width=0.75\textwidth,height=0.15\textwidth]{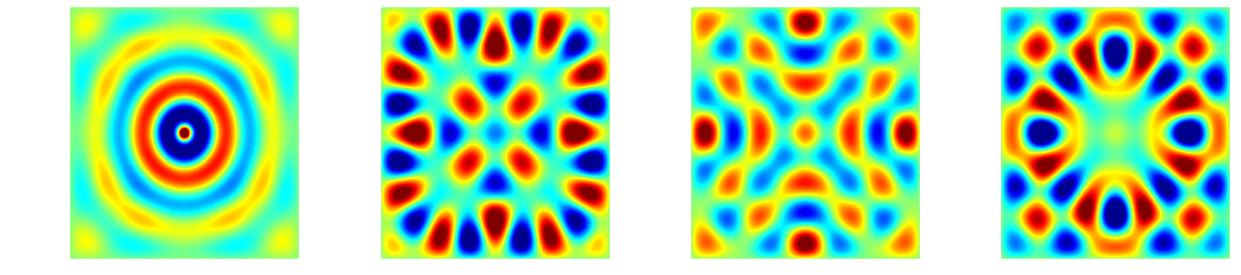} \\
\includegraphics[width=0.75\textwidth,height=0.15\textwidth]{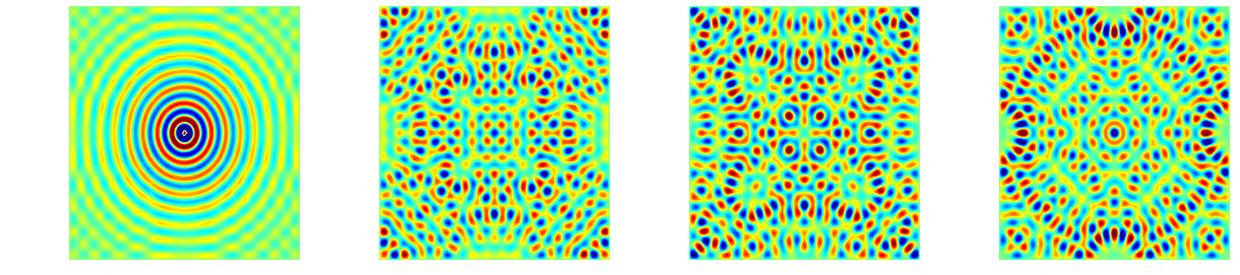} \\
\end{tabular}
\caption{Real part of the total fields for the problem considered in Example 3 with point source located at $z=(0,0)$. Upper row: $\omega_0 = 15$. Lower row: $\omega_0 = 50$. Fields at times $t=4$, $6$, $8$ and $10$ are displayed from left to right in each row.}
\label{Example3.3}
\end{figure}

\begin{figure}[htbp]
\centering
\begin{tabular}{cc}
\includegraphics[width=0.75\textwidth,height=0.15\textwidth]{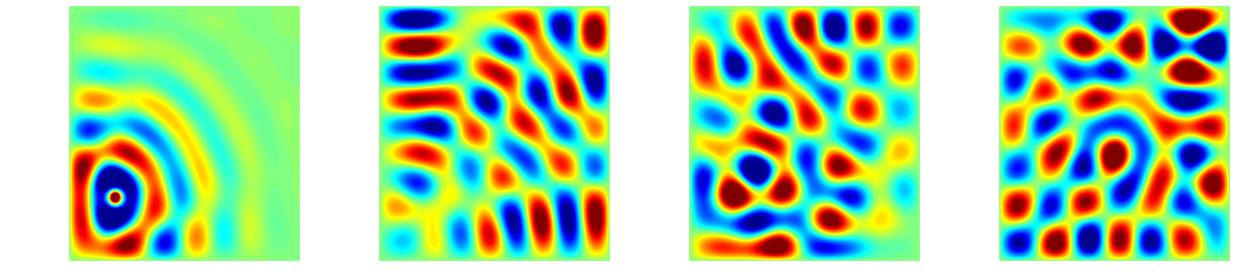} \\
\includegraphics[width=0.75\textwidth,height=0.15\textwidth]{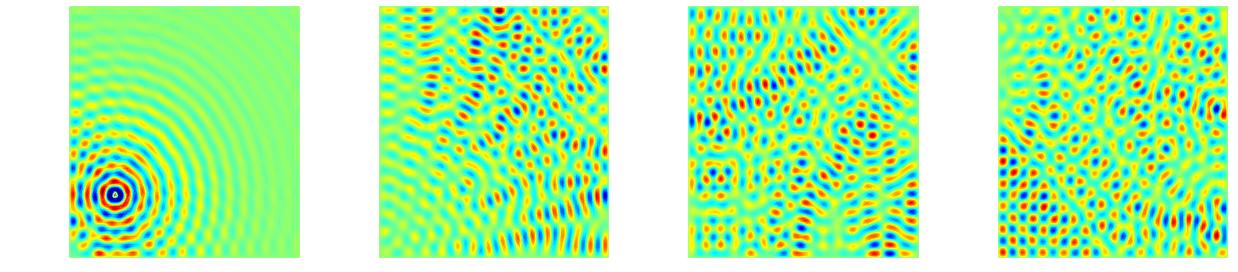} \\
\end{tabular}
\caption{Real part of the total fields for the problem considered in Example 3 with point source located at $z=(-0.6,-0.5)$. Upper row: $\omega_0 = 15$. Lower row: $\omega_0 = 50$. Fields at times $t=4$, $6$, $8$ and $10$ are displayed from left to right in each row.}
\label{Example3.4}
\end{figure}

{\bf Example 4.} We now use Algorithm~\ref{alg2} to solve wave
equation problems in a disc-shaped domain (Figure~\ref{MSmodel}(c)) and
a T-shaped domain (Figure~\ref{Example4.2}(a)), up to final times $T=10$, and using $M=7$ ping-pong iterations. The incident wave is a pulse
function given by
\begin{equation}
\label{generalinc}
u^i(x,t)=f(t-|x-z_0|/c),\quad f(s)=\sin(4s)e^{-1.6(s-3)^2},
\end{equation}
which is displayed in Figure~\ref{Example4.2}(b). Figure~\ref{Example4.2}(c) displays the corresponding
Fourier transform, in view of which the fixed numerical bandwidth
value $W=15$ was used for this example. Figures~\ref{Example4.3}-\ref{Example4.5}
display the total field within $\Omega$ at various times and for different
source point locations.

\begin{figure}[htb]
\centering
\begin{tabular}{ccc}
\includegraphics[scale=0.15]{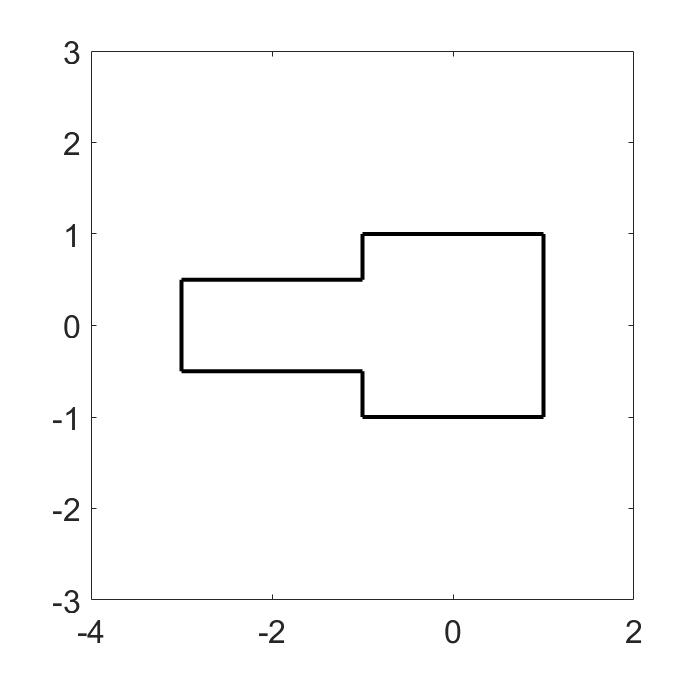} &
\includegraphics[scale=0.15]{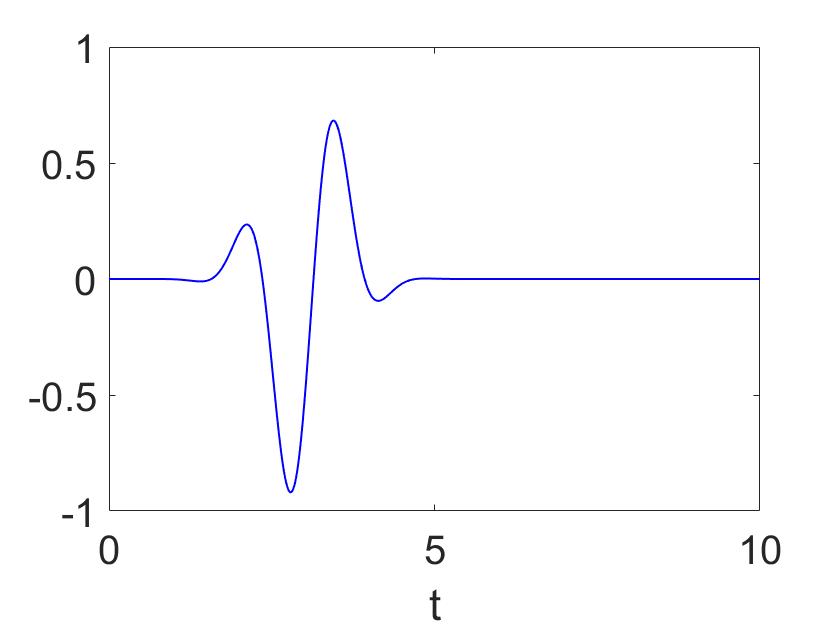} &
\includegraphics[scale=0.15]{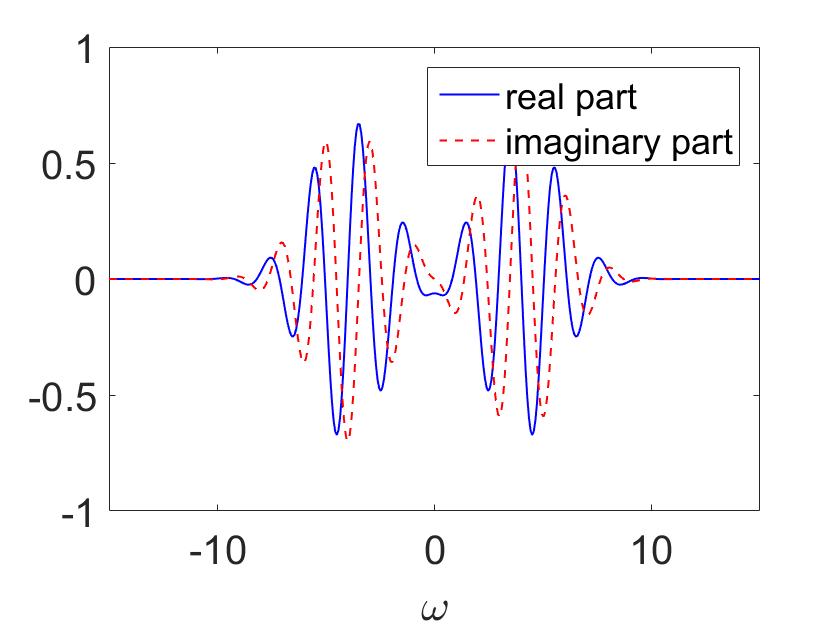} \\
(a) T-shaped domain. & (b) $u^i(x,t)$ for $|x-z_0| =1.$ & (c) $U^i(x,\omega).$
\end{tabular}
\caption{Setup utilized for the test case considered in Example 4, including,  (a) The T-shaped domain used, as well as, (b) The time-domain incident wave $u^i(x,t)$, and, (c) Its Fourier transform $U^i(x,\omega)$. The Fourier transform displayed in (c) is smaller than $10^{-8}$ outside the $\omega$-range considered in the figure.}
\label{Example4.2}
\end{figure}

\begin{figure}[htbp]
\centering
\begin{tabular}{c}
\includegraphics[width=0.75\textwidth,height=0.15\textwidth]{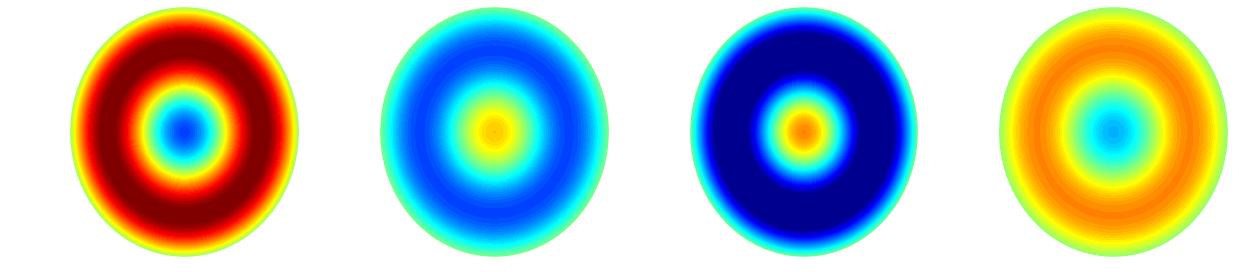} \\
\includegraphics[width=0.75\textwidth,height=0.15\textwidth]{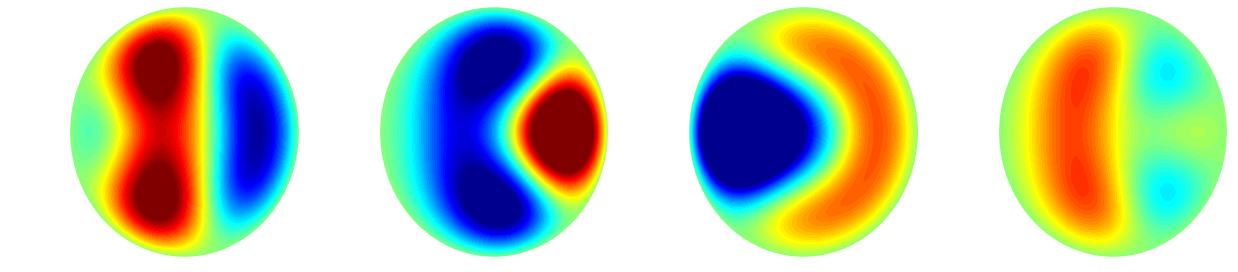}
\end{tabular}
\caption{Total fields in the disc-shaped domain considered in Example 4. Upper row: $z_0=(0,0)$. Lower row: $z_0=(-0.5,0)$. Fields at times $t=4$, $6$, $8$ and $10$ are displayed from left to right in each row.}
\label{Example4.3}
\end{figure}

\begin{figure}[htbp]
\centering
\begin{tabular}{cc}
\includegraphics[scale=0.2]{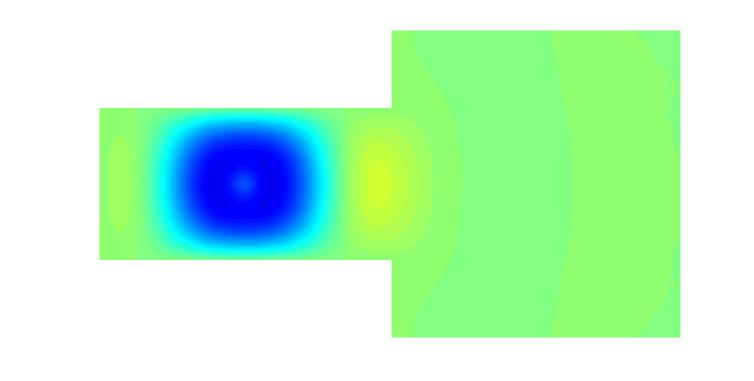} &
\includegraphics[scale=0.2]{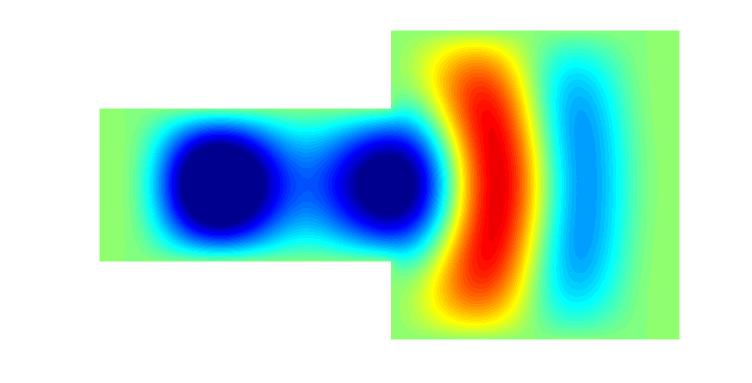} \\
(a) $t=3$ & (b) $t=5$ \\
\includegraphics[scale=0.2]{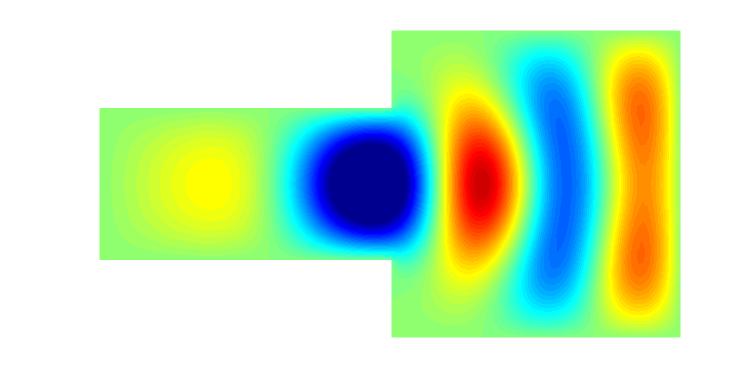} &
\includegraphics[scale=0.2]{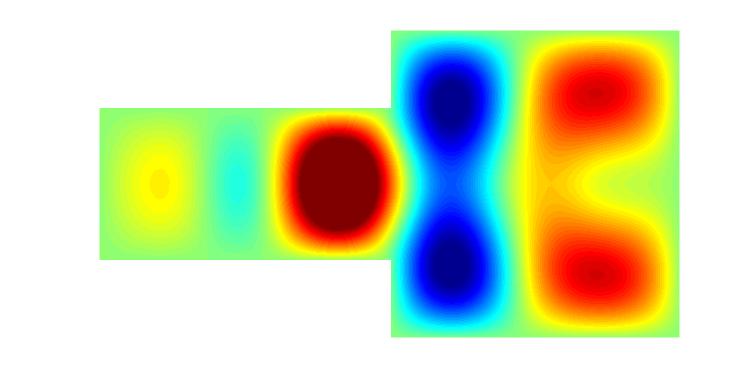} \\
(c) $t=7$ & (d) $t=9$
\end{tabular}
\caption{Total fields in the T-shaped domain considered in Example 4, with point source located at $z_0=(-2,0)$, at various times $t$.}
\label{Example4.4}
\end{figure}

\begin{figure}[htbp]
\centering
\begin{tabular}{cc}
\includegraphics[scale=0.2]{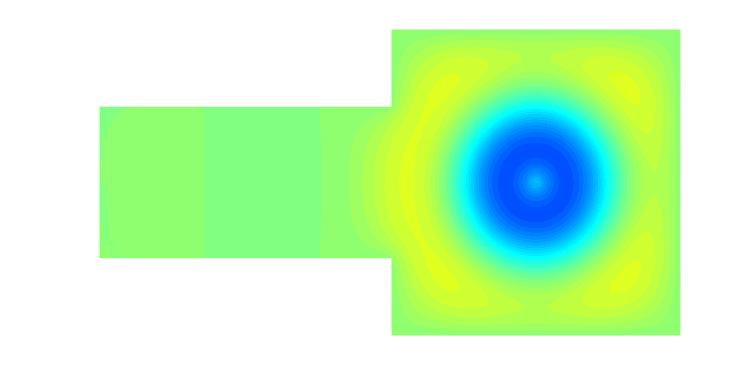} &
\includegraphics[scale=0.2]{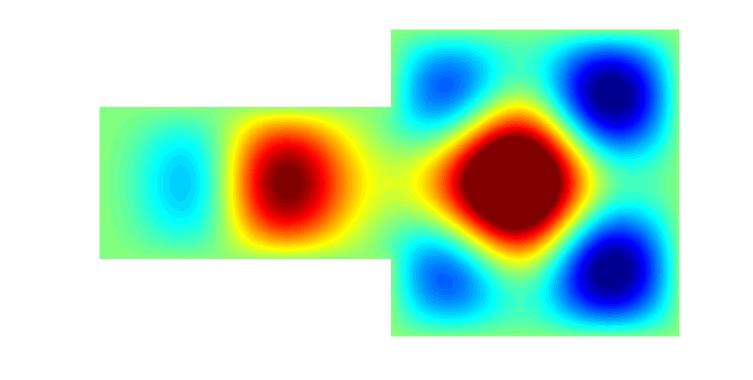} \\
(a) $t=3$ & (b) $t=5$ \\
\includegraphics[scale=0.2]{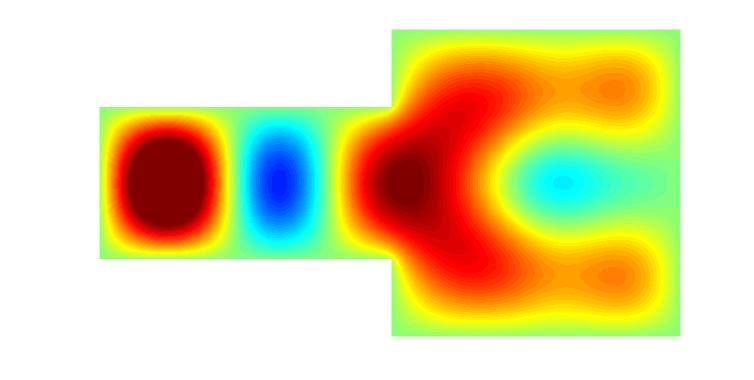} &
\includegraphics[scale=0.2]{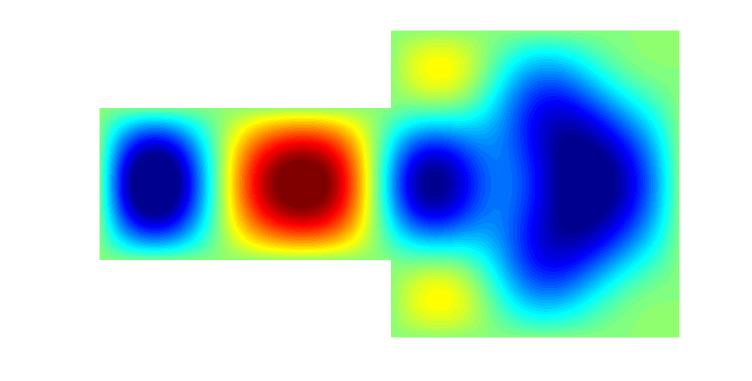} \\
(c) $t=7$ & (d) $t=9$
\end{tabular}
\caption{Total fields in the T-shaped domain considered in Example 4, with point source located at $z_0=(0,0)$, at various times $t$.}
\label{Example4.5}
\end{figure}

{\bf Example 5.} This example concerns a long time propagation and scattering problem  in a unit disc domain under the incident wave (\ref{generalinc}) with $z_0=(0,0)$. For this example we have used
$\delta_{12}=2\sin\frac{\pi}{10}\approx0.618$,
$M=45$, $K=4$ (so that $s_K+H=55$), $W=20$, $J=454$, and $\Delta t=0.11$,  and we have computed the necessary frequency domain solutions using open-arc discretizations $\mathcal{M}_j,j=1,2$, each one of which   contains 224
discretization points.  Note that the exact solution values at
the points $x=(-\sqrt{3}/4,1/4)$ and
$x=(0.5,0)$  coincide (since $|(-\sqrt{3}/4,1/4)|=|(0.5,0)|=0.5$). This simple symmetry relation provides a valuable verification of the numerical solution---which, as illustrated Figure~\ref{Example5.1}, is closely satisfied by the numerical solution. Tables~\ref{TableExp5.1} and \ref{TableExp5.2}, finally,  present the
numerical solution errors \textcolor{r1}{$\max_{t\in[0,T(M)]}|u^{s}_{\rm num}-u^{s}_{\rm ref}|$} for the present problem at the point $x = (0.5,0)$, for various values of  $M$ and corresponding final times $T(M)$, together with other
statistics such as precomputation time and total computational times. Note in particular that the solution errors do not grow as the final times increase.

\begin{figure}[htbp]
\centering
\includegraphics[scale=0.25]{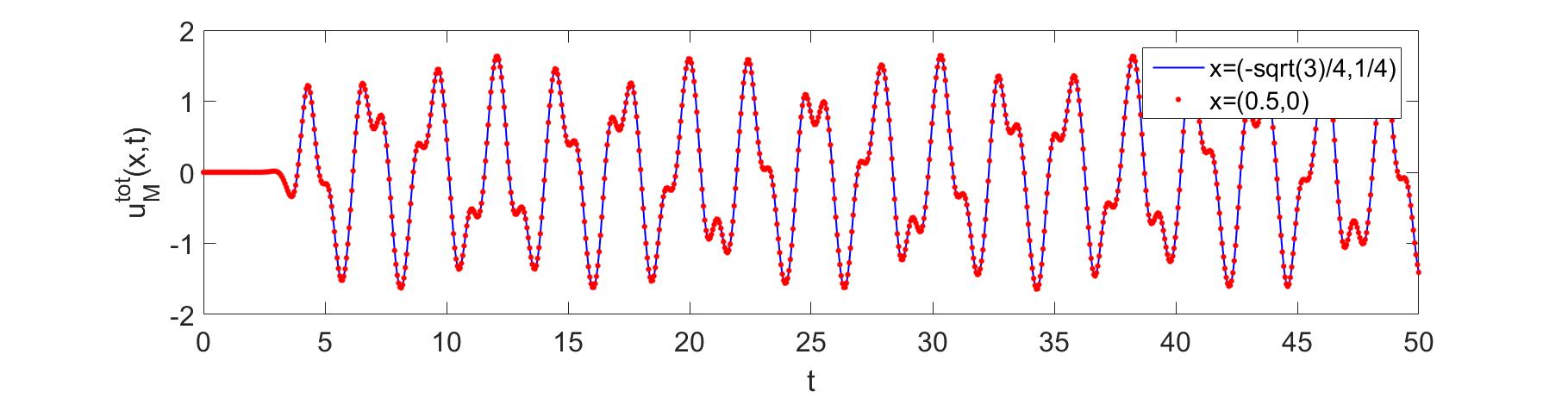}
\caption{Time-domain solutions $u_{M}^{tot}(x,t), t\in[0,50]$ considered in Example 5 at $x=(-\sqrt{3}/4,1/4)^\top$ and $x=(0.5,0)^\top$ with $M=45$. As illustrated in the figure, these two \textcolor{r1}{functions} coincide, by symmetry.}
\label{Example5.1}
\end{figure}


\begin{table}[htb]
  \caption{Numerical errors, precomputation time and total computational times required by the problem considered in Example 5 for various values of $M$. $J=254$ frequencies were used in all cases.}
\centering
\begin{tabular}{|c|c|c|c|c|c|}
\hline
$M$ & 15 &  25 & 35 & 45 \\
\hline
$T(M)$ & 8.652 &  14.832 & 21.012 & 27.192 \\
\hline
Error & $6.57\times 10^{-4}$ & $4.74\times 10^{-5}$ & $1.87\times 10^{-5}$ & $5.56\times 10^{-6}$ \\
\hline
Time (precomputation) & \multicolumn{4}{|c|}{9.1 s} \\
\hline
Time ($M$ iterations) & 30.0 s  & 52.2 s & 73.1 s & 94.4 s \\
\hline
\end{tabular}
\label{TableExp5.1}
\end{table}

\begin{table}[htb]
  \caption{Numerical errors, precomputation time and total computational times required by the problem considered in Example 5 for various values of $M$. $J=454$ frequencies were used in all cases.}
\centering
\begin{tabular}{|c|c|c|c|c|c|}
\hline
$M$ & 15 &  25 & 35 & 45\\
\hline
$T(M)$ & 8.652 &  14.832 & 21.012 & 27.192 \\
\hline
Error & $1.24\times 10^{-7}$ & $2.59\times 10^{-8}$ & $1.10\times 10^{-8}$ & $4.36\times 10^{-9}$ \\
\hline
Time (precomputation) & \multicolumn{4}{|c|}{16.8 s} \\
\hline
Time ($M$ iterations) & 46.3 s  & 78.4 s & 109.8 s & 141.3 s \\
\hline
\end{tabular}
\label{TableExp5.2}
\end{table}

\textcolor{r1}{{\bf Example 6.} In our final example we briefly demonstrate the feasibility of a version of the proposed multiple scattering algorithm which utilizes more than two patches. This extended algorithm requires use of appropriately windowed boundary data for the open-surface wave equation problems associated with multiple patches. At each step, the multiple wave equation problems can be solved in parallel. Complete details concerning the algorithm and its implementation will be presented elsewhere~\cite{BBY}. In the example presented here a three-patch decomposition of the boundary $\Gamma$, as shown in Figure~\ref{Example7.1}, is utilized to solve once again the wave equation problem (\ref{waveeqn}) on the unit disc, under plane-wave  incidence $u_2^i(x,t)$, considered in Example 2, and the exact solution is given by
$u^s(x,t)=-u_2^i(x,t)$ for $x\in\Omega$. The numerical solutions at $x=(0.5,0)$ as functions of time $t$ for various values of $M$ are displayed in Figure~\ref{Example7.2}. The maximum numerical errors are of the same order as those displayed in Figure~\ref{Example3.2}(b) for the two-patch case.}

\begin{figure}[htbp]
\centering
\includegraphics[scale=0.15]{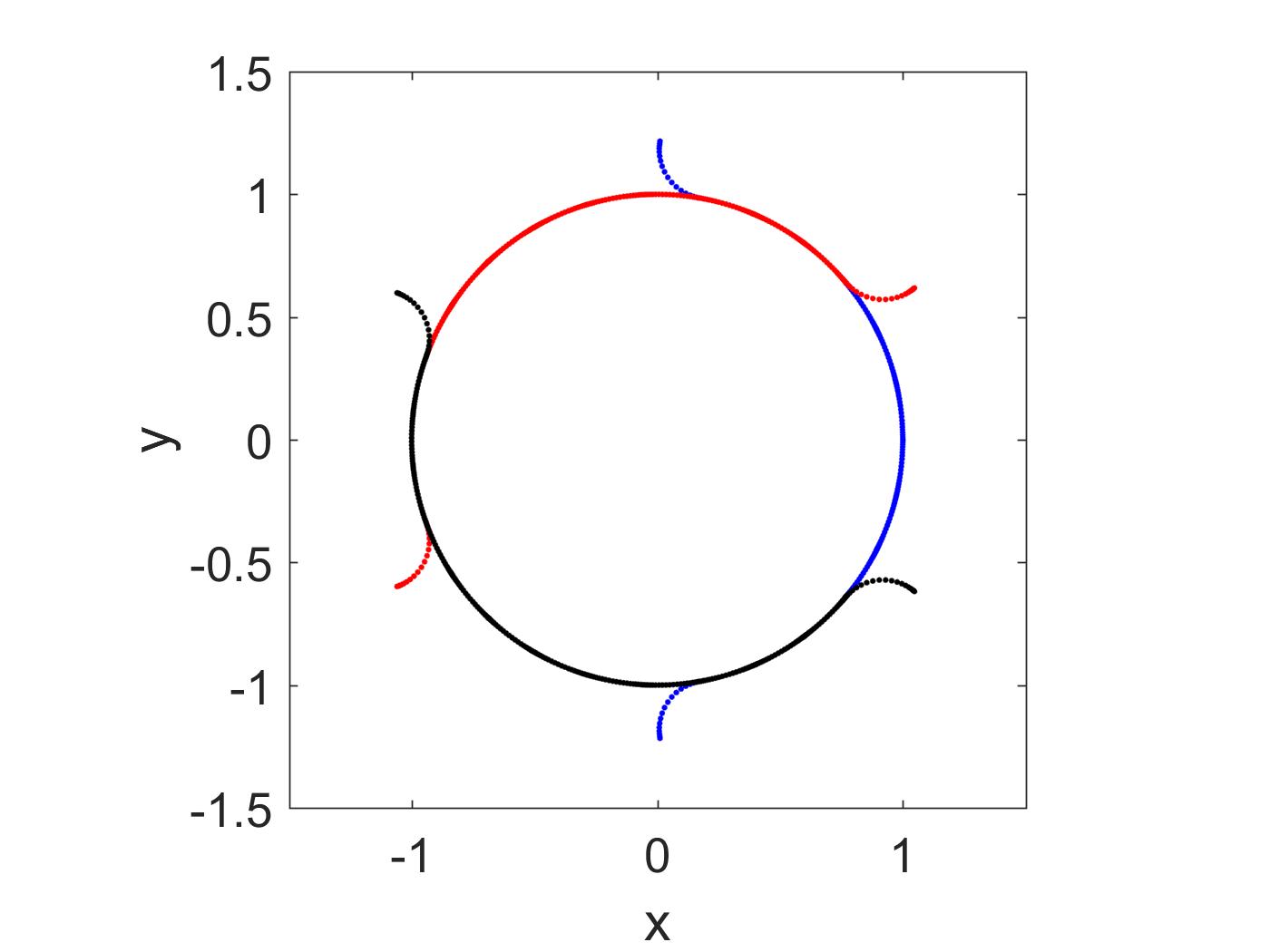}
\caption{Decomposition of a circular closed curve using three overlapping patches with extension.}
\label{Example7.1}
\end{figure}

\begin{figure}[htbp]
\centering
\begin{tabular}{ccc}
\includegraphics[scale=0.1]{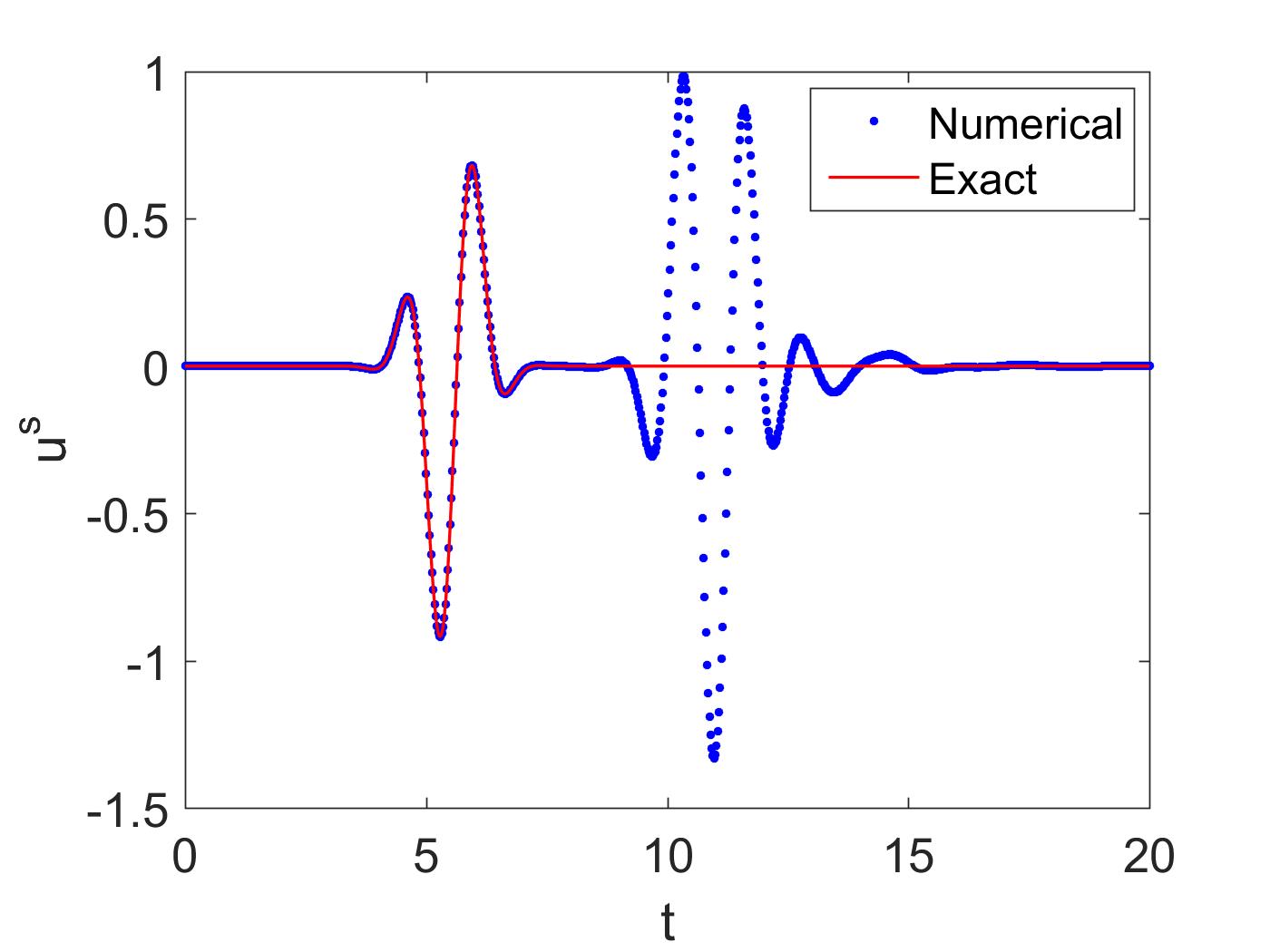} &
\includegraphics[scale=0.1]{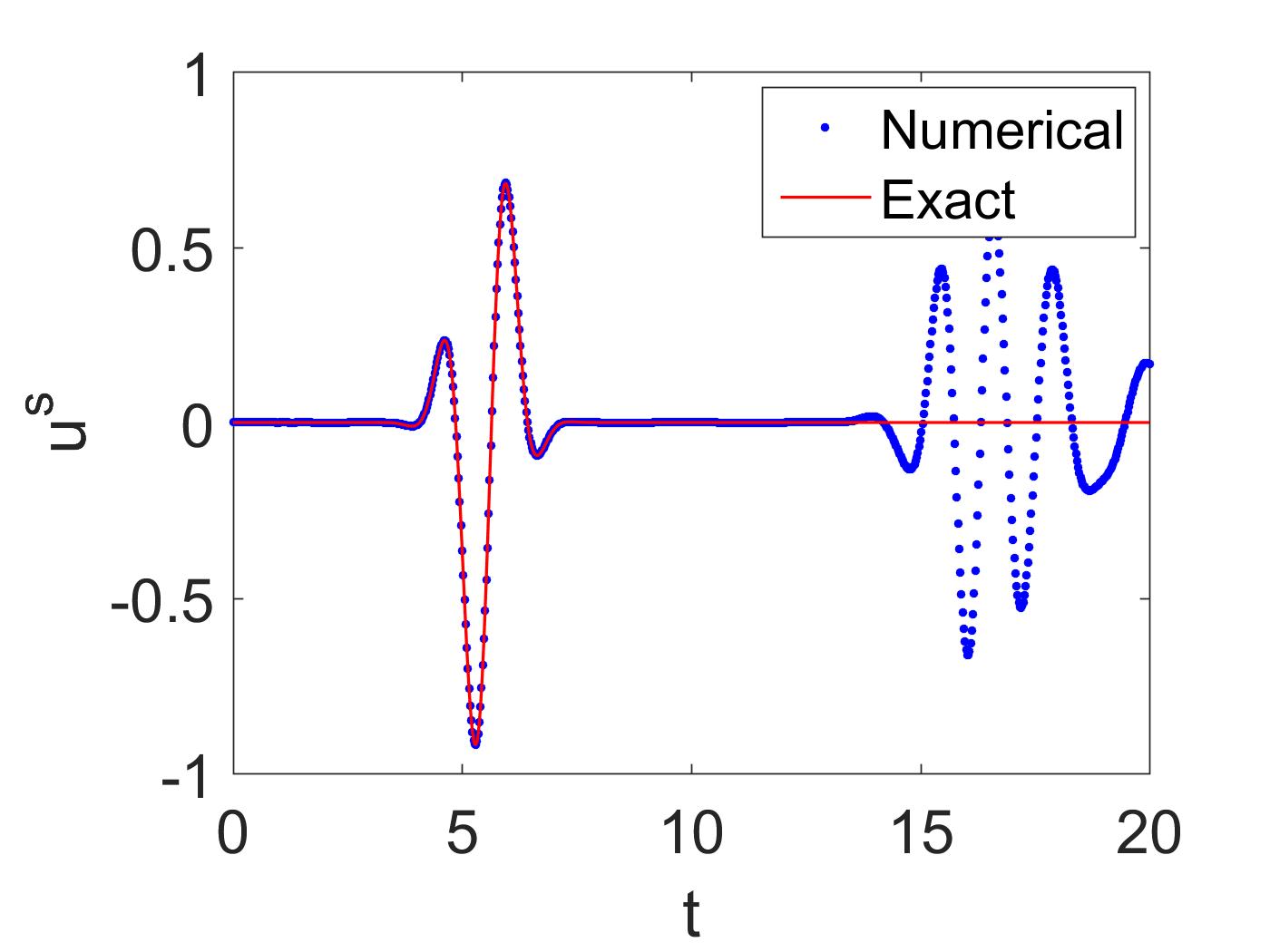} &
\includegraphics[scale=0.1]{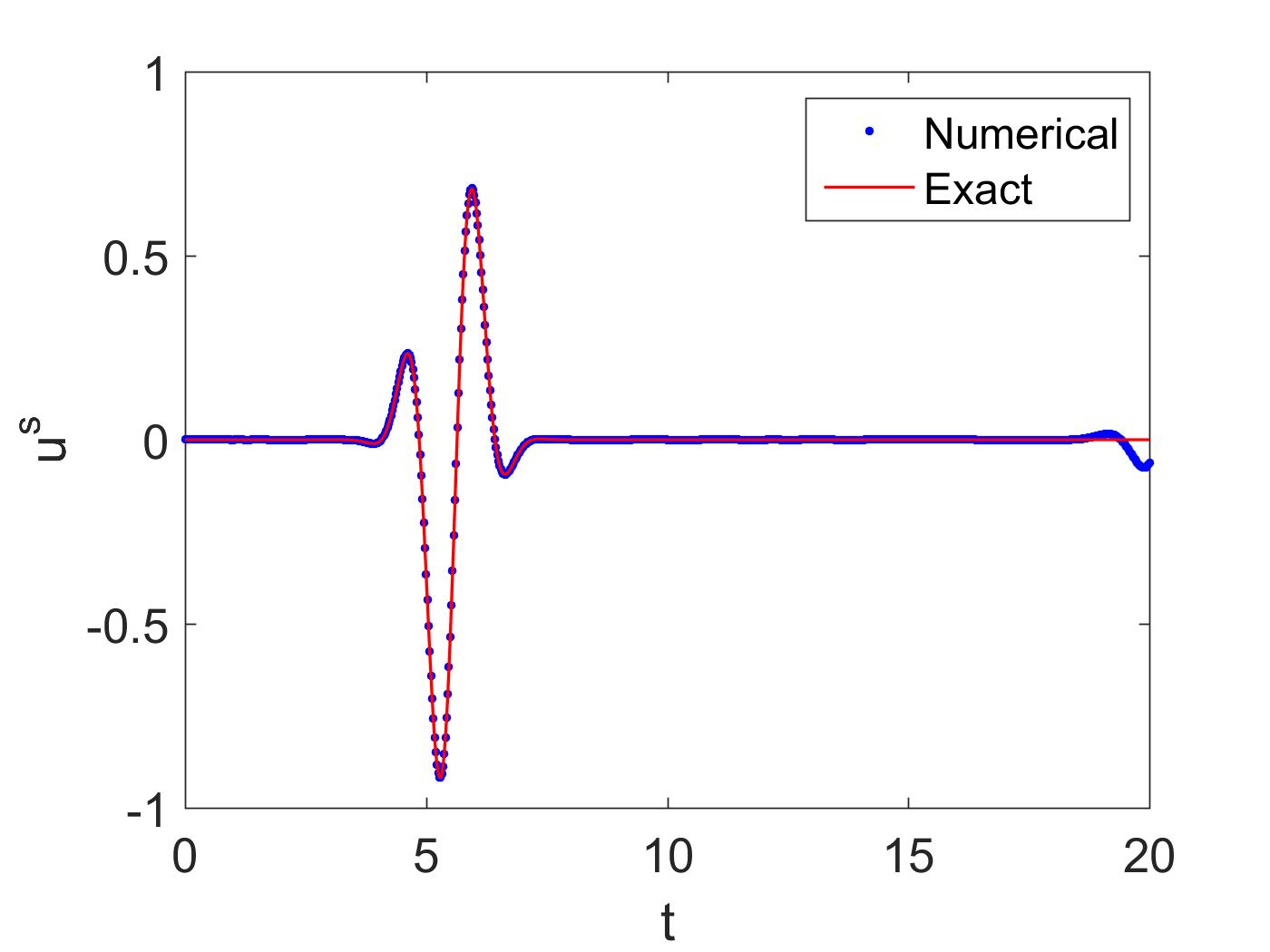} \\
(a) $M=3$ & (b) $M=6$ & $M=9$
\end{tabular}
\caption{Comparison of the numerical and exact solutions as functions of time $t$ for various values of $M$.}
\label{Example7.2}
\end{figure}

\section{Conclusions}
\label{sec:5}

This paper proposed a frequency-time hybrid integral-equation method
for the wave equation problem in an interior two-dimensional bounded
spatial domain. The solver is based on a novel ping-pong multiple
scattering strategy that reduces the original problem to a sequence of
problems of scattering by open-arcs. Exploiting the Huygens
principle, relying on a domain decomposition strategy based on use of
overlapping patches, and utilizing boundary integral equation
formulations for frequency-domain sub-problems and an efficient
Fourier transform algorithm, the proposed method produces the interior
time-domain solution efficiently and with high accuracy. An extension
of the ping-pong algorithm that incorporates arbitrary numbers of
overlapping subdomains should enable application of the method to
complex 2D and 3D geometries. The method can also be extended to
enable solution of elastic and electromagnetic wave
problems, and including problems of scattering by impenetrable
obstacles, \textcolor{r1}{problems of transmission for penetrable structures and problems in multi-layered media}. Such extensions, which lie beyond the scope of this paper, are left for future work.

\section*{Acknowledgments}
OB gratefully acknowledges support from NSF, DARPA and AFOSR through contracts
DMS-2109831, HR00111720035, FA9550-19-1-0173 and FA9550-21-1-0373,
and by the NSSEFF Vannevar Bush Fellowship under contract number
N00014-16-1-2808. TY gratefully
acknowledges support from NSFC through \textcolor{r1}{Grants No. 12171465 and 12288201}.

\appendix
\section{Appendix}
\label{sec:A}

\renewcommand{\theequation}{A.\arabic{equation}}
\renewcommand{\thetheorem}{A.\arabic{theorem}}

\textcolor{r1}{This appendix obtains the explicit expression
  (\ref{TDsol1}), used in the proof of Lemma~\ref{lemma_circle}, for
  the solution \textcolor{r1}{$v_+^s(x,t)$} of problem
  (\ref{waveeqn-ballTD}), where $g$ is a causal function ($g(x,t)=0$
  for $t\leq 0$) defined on $\mathbb{R}_0^2\times\mathbb{R}$, which
  satisfies the assumption~\eqref{DoI_ball_ass} for a bounded subset
  $\mathcal{C}^{\mathrm{inc}}\subset \mathbb{R}_0^2$. The construction
  utilizes the associated frequency-domain Green's function
  $G_\omega(x,y)$, which, for each $y\in\R^2_+$, is defined as the
  solution of the problem \ben
\begin{cases}
\Delta_x G_\omega(x,y)+\kappa^2G_\omega(x,y)=-\delta_y(x), & (x,\omega)\in\R^2_+\times\R, \ \kappa=\omega/c, \cr
G_\omega(x,y)= 0  & (x,\omega)\in\R_0\times\R.
\end{cases}
\enn
As is well known, the method of images yields
\be
\label{GreensFunc}
G_\omega(x,y)=\frac{i}{4}H_{0}^{(1)}(\kappa|x-y|)-
\frac{i}{4}H_{0}^{(1)}(\kappa|x'-y|),\quad x\ne y, \en where
$x'=(x_1,-x_2)$ denotes the image point of $x=(x_1,x_2)\in\R^2_+$ with
respect to $\R^2_0$, and where $H_0^{(1)}$ denotes the Hankel function
of first kind and order zero. Let now $V_+^s(x,\omega)$ denote the
Fourier transform of $v_+^s(x,t)$ with respect to $t$ for
$x\in\R^2_+$; clearly $V_+^s$ is a solution of the Helmholtz equation
with wavenumber $\kappa$ in $\R^2_+$, and with Dirichlet boundary
conditions $V_+^s=\widehat{g}$ on $\R^2_0$---where
$\widehat{g}(x,\omega)$ denotes the Fourier transform of $g(x,t)$ with
respect to~$t$. Clearly $\widehat{g}(x,\omega)$ vanishes for
$x\not\in \mathcal{C}^{\mathrm{inc}}$ since, in view
of~\eqref{DoI_ball_ass}, so does $g(x,t)$.  Use of Green's theorem together with the
Green function $G_\omega$ \cite{DM97} yields \be
\label{FDsol}
V_+^s(x,\omega)=\int_{\mathcal{C}^{\mathrm{inc}}}
\pa_{\nu_y}G_\omega(x,y)\widehat{g}(y,\omega)ds_y,\quad x\in\R^2_+,
\en where, denoting by $\nu_y=(0,1)^\top$ the unit upward normal on
$\R^2_0$, \ben
\textcolor{r1}{\pa_{\nu_y}G_\omega(x,y)=\frac{i\kappa}{2}\frac{x_2}{|x-y|}
  H_{1}^{(1)}(\kappa|x-y|).}  \enn It follows that, for $x\in \R^2_+$,
$v_+^s(x,t)$ equals the inverse Fourier transform of (\ref{FDsol}). To
proceed with the construction we introduce the following notations.}
\textcolor{r1}{We call $\R^3_0=\R^2_0\times \R$ the plane in three dimensional
  space with cross-section $\R^2_0$, we let
  $\breve{x}=(x^\top,0)^\top\notin\R^3_0$,
  $\breve{y}=(y^\top,z)^\top\in\R^3_0$ and $\breve{r}=\sqrt{r^2+z^2}$
  with $r=|x-y|=\sqrt{(x_1-y_1)^2+(x_2-y_2)^2}$. Using the unit normal
  $\nu_y=(0,1)^\top$ of $\R^2_0$, finally, the corresponding unit
  normal on $\R^3_0$ is denoted by $\nu_{\breve{y}}=(0,1,0)^\top$.}

In preparation for the main result of this appendix we establish the following Lemma.
\begin{lemma}
\label{rest_lemma}
The following formulas hold:
\be
\label{formula1}
&& \frac{1}{4\pi}\int_{-\infty}^\infty \frac{e^{i\kappa \breve{r}}}{ \breve{r}}dz= \frac{i}{4}H_{0}^{(1)}(\kappa r),\\
\label{formula2}
&& \frac{1}{4\pi}\int_{-\infty}^\infty \pa_{y_i}\left(\frac{e^{i\kappa \breve{r}}}{\breve{r}}\right)dz= \frac{i}{4} \pa_{y_i} H_{0}^{(1)}(\kappa r),\quad i=1,2.
\en
\end{lemma}
\begin{proof}
The expression (\ref{formula1}) is established in \cite[Lemma 3.1]{CF91}. Using the notations
\ben
\breve{r}_0=(x_1^2+(x_2-y_2)^2+z^2)^{1/2},
\enn
to establish the $y_1$ component of~(\ref{formula2}) (the $y_2$ component follows analogously), it suffices to show that
\be
\label{formula3}
\int_{0}^{y_1}dy_1  \int_{1}^\infty \pa_{y_1} \left(\frac{e^{i\kappa \breve{r}}}{\breve{r}}\right) dz = \int_{1}^\infty \frac{e^{i\kappa \breve{r}}}{\breve{r}}dz- \int_{1}^\infty \frac{e^{i\kappa \breve{r}_0}}{\breve{r}_0}dz;
\en
the result then follows by differentiation with respect to $y_1$. To establish (\ref{formula3}), we seek to utilize \textcolor{r1}{Fubini's} Theorem on the left hand integral, but,  unfortunately, the integrand does not satisfy the hypothesis of Fubini's theorem: it is not an integrable function of the variable $(y_1,z)$. To address this difficulty we integrate by parts the left-hand integral: using the relation
\be
\label{formula4}
e^{i\kappa \breve{r}}=\frac{\breve{r}}{i\kappa z}\pa_z e^{i\kappa \breve{r}},
\en
 we obtain
\ben
\int_{1}^\infty \pa_{y_1} \left(\frac{e^{i\kappa \breve{r}}}{\breve{r}}\right) dz= \int_{1}^\infty \frac{1}{i\kappa z}\pa_z\pa_{y_1} e^{i\kappa \breve{r}} dz= -\frac{1}{i\kappa}\pa_{y_1} e^{i\kappa \breve{r}}\Big|_{z=1}+ \int_{1}^\infty\frac{1}{i\kappa z^2} \pa_{y_1} e^{i\kappa \breve{r}} dz.
\enn
\textcolor{r1}{Fubini's} Theorem can now be applied to \textcolor{r1}{the last integral}, and we thus obtain
\ben
\int_{0}^{y_1}dy_1  \int_{1}^\infty \pa_{y_1} \left(\frac{e^{i\kappa \breve{r}}}{\breve{r}}\right) dz &=& - \frac{1}{i\kappa}e^{i\kappa \breve{r}}\Big|_{z=1}+ \frac{1}{i\kappa}e^{i\kappa \breve{r}_0}\Big|_{z=1}+ \int_{0}^{y_1}dy_1 \int_{1}^\infty\frac{1}{i\kappa z^2} \pa_{y_1} e^{i\kappa \breve{r}} dz\\
&=& - \frac{1}{i\kappa}e^{i\kappa \breve{r}}\Big|_{z=1}+ \frac{1}{i\kappa}e^{i\kappa \breve{r}_0}\Big|_{z=1}+ \int_{1}^\infty\frac{1}{i\kappa z^2}  (e^{i\kappa \breve{r}}-e^{i\kappa \breve{r}_0}) dz.
\enn
Now, replacing $z^{-2} = - \pa_z z^{-1}$ and integrating by parts in the last integral,  and then using   (\ref{formula4}) once again, equation (\ref{formula3}) results, as desired. The proof is now complete.
\end{proof}

To establish (\ref{TDsol1}) we proceed as follows. In view of Lemma~\ref{rest_lemma} and equation (\ref{GreensFunc}), and noting that for \textcolor{r1}{$x\in\R^2_+, y\in\R^2_0$ we have $x_2-y_2\ne0$, the normal derivative of the Green function on the boundary $\R^2_0$ is given by
\ben
\pa_{\nu_y}G_\omega(x,y) &=& \frac{i\kappa}{2}\frac{x_2}{r} H_{1}^{(1)}(\kappa r)=\frac{i}{2} \pa_{y_2} H_{0}^{(1)}(\kappa r) =\frac{1}{2\pi}\int_{-\infty}^\infty \pa_{y_2}\left(\frac{e^{i\kappa \breve{r}}}{\breve{r}}\right)dz =-\frac{1}{2\pi} \int_{-\infty}^\infty \frac{i\kappa \breve{r}-1}{\breve{r}^3} x_2 e^{i\kappa \breve{r}}dz
\enn
and, therefore, equation (\ref{FDsol}) gives
\ben
V_+^s(x,\omega)= -\frac{1}{2\pi} \int_{\mathcal{C}^{\mathrm{inc}}\times \R} \frac{i\kappa \breve{r}-1}{\breve{r}^3} x_2e^{i\kappa \breve{r}}\, \widehat{g}(y,\omega)ds_{\breve{y}},\quad x\in\R^2_+.
\enn
Taking the inverse Fourier transform we obtain
\ben
v_+^s(x,t)&=& \frac{1}{2\pi} \int_{\mathcal{C}^{\mathrm{inc}}\times \R} \int_{0}^t \left[ \frac{1}{\breve{r}^3}\delta(t-\tau-c^{-1}\breve{r}) + \frac{1}{c\breve{r}^2} \delta'(t-\tau-c^{-1}\breve{r}) \right]x_2\, g(y,\tau)d\tau ds_{\breve{y}} \\
 &=& \frac{1}{2\pi} \int_{\mathcal{C}^{\mathrm{inc}}\times \R} \left[ \frac{1}{\breve{r}^3} g(y,t-c^{-1}\breve{r})+ \frac{1}{c\breve{r}^2} g^{(1)}(y,t-c^{-1}\breve{r})\right]x_2ds_{\breve{y}}.
\enn
where $g^{(1)}(x,t)=\frac{\pa g}{\pa t}(x,t)$. Using the relations $ds_{\breve{y}}=dzds_y$ we thus obtain
\ben
v_+^s(x,t)
&=& \frac{1}{\pi} \int_{\mathcal{C}^{\mathrm{inc}}} \int_{0}^{+\infty} \Big[(r^2+z^2)^{-3/2} g(y,t-c^{-1}\sqrt{r^2+z^2})\\
&&\quad\quad\quad +\frac{1}{c(r^2+z^2)} g^{(1)}(y,t-c^{-1}\sqrt{r^2+z^2}) \Big]x_2 dzds_y.
\enn
Utilizing the change of variables $\tau=t-c^{-1}\sqrt{r^2+z^2}$, or equivalently $z=c\sqrt{(t-\tau)^2-c^{-2}r^2}$, and $dz=-cz^{-1}\sqrt{r^2+z^2}d\tau$, we then obtain
\ben
(r^2+z^2)^{-3/2}dz=-\frac{1}{c^2(t-\tau)^2\sqrt{(t-\tau)^2-c^{-2}r^2}}d\tau
\enn
and
\ben
\frac{1}{c(r^2+z^2)}dz =-\frac{1}{c^2(t-\tau)\sqrt{(t-\tau)^2-c^{-2}r^2}}d\tau.
\enn
It then follows that, for $x\in\R^2_+$,
\begin{align}
\label{appendx_solrep}
v_+^s&(x,t)=\frac{1}{\pi c^2} \\
& \int_{\mathcal{C}^{\mathrm{inc}}}\int_{0}^{t-c^{-1}|x-y|} \left[\frac{x_2\, g(y,\tau)} {(t-\tau)^2\sqrt{(t-\tau)^2-c^{-2}|x-y|^2}} +\frac{x_2\,g^{(1)}(y,\tau)} {(t-\tau)\sqrt{(t-\tau)^2-c^{-2}|x-y|^2}} \right] d \tau ds_y,
\end{align}
as desired.}

\begin{remark}
\label{double_layer}
Although not used in this paper, it is worthwhile to note here that,
for an {\em arbitrary} two-dimensional Lipschitz curve \textcolor{r1}{$\Gamma$}, and
letting $\breve{\Gamma}=\Gamma\times \R$, the \textcolor{r1}{changes} of variables
used in this appendix can easily be utilized to obtain an expression
for the time-domain double-layer potential in two-dimensions, which
had heretofore not been successfully derived. Utilizing \textcolor{r1}{once again the
notations used} above, from the Kirchhoff formula~\cite[Eq. (22),
Sec. 8.1]{S41}, we know that the classical three-dimensional
time-domain double-layer potential is given by \ben
\mathcal{D}_{3D}(\phi)(\breve{x},t)= \frac{1}{4\pi}
\int_{\breve{\Gamma}} \left[\frac{\pa
    (\breve{r}^{-1})}{\pa\nu_{\breve{y}}}
  \phi(\breve{y},t-c^{-1}\breve{r})- \frac{1}{c\breve{r}} \frac{\pa
    \breve{r}}{\pa\nu_{\breve{y}}}
  \phi^{(1)}(\breve{y},t-c^{-1}\breve{r})\right]ds_{\breve{y}}.  \enn
The two-dimensional double-layer potential $\mathcal{D}_{2D}$ can then be obtained by assuming that the causal signal $\phi$ is independent
of $z$. Then using the change of variables
$\tau=t-c^{-1}\sqrt{r^2+z^2}$ we obtain \ben
&&\mathcal{D}_{2D}(\phi)(x,t)\\
&&= \frac{1}{2\pi c^2} \int_{\Gamma} \int_{0}^{t-c^{-1}|x-y|}
\left[\frac{(x-y)\cdot\nu_y\phi(y,\tau)}{(t-\tau)^2\sqrt{(t-\tau)^2-c^{-2}|x-y|^2}}
  +\frac{(x-y)\cdot\nu_y\phi^{(1)}(y,\tau)}{(t-\tau)\sqrt{(t-\tau)^2-c^{-2}|x-y|^2}}
\right] d\tau ds_y.  \enn which provides a correction to the
expression presented in \cite[Page 19]{S16}. The contributions
\cite[Sections 6.3-6.5]{BC39}, \cite[Page 42]{J64} and references
therein outline some of the difficulties previously encountered in
regard to the 2D double-layer potential.
\end{remark}


\begin{thebibliography}{00}
\bibitem{ADG11} \textcolor{r1}{A. Aimi, M. Diligenti, C. Guardasoni, On the energetic Galerkin boundary element method applied to interior wave propagation problems, J. Comput. Appl. Math. 235(7) (2011) 1746-1754.}
\bibitem{AGH02} B. Alpert, L. Greengard, T. Hagstrom, Nonreflecting boundary conditions for the time-dependent wave equation, J. Comput. Phys. 180 (2002) 270-296.
\bibitem{AB16} F. Amlani, O.P. Bruno, An FC-based spectral solver for elastodynamic problems in general three-dimensional domains, J. Comput. Phys. 307 (2016) 333-354.
\bibitem{At91} K.E. Atkinson, I.H. Sloan. The numerical solution of
  first-kind logarithmic-kernel integral equations on smooth open
  arcs. Math. Comp.  56 (1991) 119-139.
\bibitem{A20} T.G. Anderson, Hybrid Frequency-Time Analysis and Numerical Methods for Time-Dependent Wave Propagation (Ph.D. thesis), California Institute of Technology, 2020.
\bibitem{ABL18} T.G. Anderson, O.P. Bruno, M. Lyon, High-order, dispersionless ``fast-hybrid'' wave equation solver. Part I: O(1) sampling cost via incident-field windowing and recentering, SIAM J. Sci. Comput. 422 (2020) A1348-A1379.
\bibitem{BS97} I.M. Babu\v{s}ka, S.A. Sauter, Is the pollution effect of the FEM avoidable for the Helmholtz equation considering high wave numbers? SIAM J. Numer. Anal. 34 (1997) 2392-2423.
\bibitem{BC39} B.B. Baker, E.T. Copson, The Mathematical Theorey Huygens' Principle, Oxford at the Clarendon Press, 1939.
\bibitem{BH86} A. Bamberger, T. Ha-Duong, Formulation variationelle espace-temps pour le calcul par potentiel retard\'e de la diffraction d'une onde acoustique, Math. Methods Appl. Sci. 8 (1986) 405-435.
\bibitem{BK14} L. Banjai, M. Kachanovska, Fast convolution quadrature for the wave equation in three dimensions, J. Comput. Phy. 279 (2014) 103-126.
\bibitem{BMPS21} \textcolor{r1}{P. Bansal, A. Moiola, I. Perugia, C. Schwab, Space-time discontinuous Galerkin approximation of acoustic waves with point singularities, IMA J. Numer. Anal. 41 (2021) 2056-2109.}
\bibitem{BBY} \textcolor{r1}{G. Bao, O.P. Bruno, T. Yin, Multiple-scattering frequency-time hybrid integral equation solver for the wave equation problems with bounded obstacles, in preparation.}
\bibitem{BGH20} A.H. Barnett, L. Greengard, T. Hagstrom, High-order discretization of a stable time-domain integral equation for 3D acoustic scattering, J. Comput. Phy. 402 (2020) 109047.
\bibitem{BB21} C. Bauinger, O.P. Bruno, ``Interpolated Factored Green Function'' method for accelerated solution of scattering problems, J. Comput. Phy. 430 (2021) 110095.
\bibitem{B94} J.-P. Berenger, A perfectly matched layer for the absorption of electromagnetic waves, J. Comput. Phy. 114 (1994) 185-200.
\bibitem{Betcke:17} T.~Betcke, N.~Salles, and W.~\'Smigaj, Overresolving in the laplace domain for convolution quadrature methods, SIAM J. Sci. Comput. 39 (2017),
pp.~A188--A213.
\bibitem{BLPT16} O.P. Bruno, M. Lyon, C. P\'erez-Arancibia, C. Turc, Windowed Green function method for layered-media scattering, SIAM J. Appl. Math. 76(5) (2016) 1871-1898.
\bibitem{BG18} O.P. Bruno, E. Garza, A Chebyshev-based rectangular-polar integral solver for scattering by general geometries described by non-overlapping patches, J. Comput. Phys. 421 (2020) 109740.
\bibitem{BGP17} O.P. Bruno, E. Garza, C. P\'erez-Arancibia, Windowed Green function method for nonuniform open-waveguide problems, IEEE Trans. Antenn. Propag. 65 (2017) 4684-4692.
\bibitem{BL13} O.P. Bruno, S. Lintner, A high-order integral solver for scalar problems of diffraction by screens and apertures in three-dimensional space, J. Comput. Phys.  252 (2013).
\bibitem{BL12} O.P. Bruno, S. Lintner, Second-kind integral solvers for TE and TM problems of diffraction by open arcs, Radio Sci. 47 (6) (2012).
\bibitem{BK01} O.P. Bruno, L. Kunyansky, A fast, high-order algorithm for the solution of surface scattering problems: Basic implementation, tests, and applications, J. Comput. Phys. 169 1 (2001) 80-110.
\bibitem{BL10} O.P. Bruno, M. Lyon, High-order unconditionally stable FC-AD solvers for general smooth domains I. Basic elements, J. Comput. Phys. 229 (6) (2010) 2009-2033.
\bibitem{BY20} O.P. Bruno, T. Yin, Regularized integral equation methods for the elastic scattering problems in three dimensions, J. Comput. Phys. 410 (2020) 109350.
\bibitem{BY21} O.P. Bruno, T. Yin, A windowed Green function method for elastic scattering problems on a half-space, Comput. Methods Appl. Mech. Engrg. 376 (2021) 113651.
\bibitem{CHLM10} Q. Chen, H. Haddar, A. Lechleiter, P. Monk, A sampling method for inverse scattering in the time domain, Inverse Problems 8 (2010) 85001-85017.
\bibitem{CF91} X. Chen, A. Friedman, Maxwell's equations in a periodic structure, Transactions of the American Mathematical Society 323 (1991) 465-507.
\bibitem{CK98} D. Colton and R. Kress, Inverse Acoustic and Electromagnetic Scattering Theory, Berlin, Springer, 1998.
\bibitem{CDD03} M. Costabel, M. Dauge, R. Duduchava, Asymptotics without logarithmic terms for crack problems, Commun. Partial Differ. Equ. 28 (2003) 869-926.
\bibitem{DM97} \textcolor{r1}{J. DeSanto, P.A. Martin, On the derivation of boundary integral equations for scattering by an infinite one-dimensional rough surface, J. Acous. Soc. Am. 102(1) (1997) 67-77.}
\bibitem{DSSB93} J. Douglas, J.E. Santos, D. Sheen, L.S. Bennethum, Frequency domain treatment of one-dimensional scalar waves, Math. Models Methods Appl. Sci. 3 (1993) 171-194.
\bibitem{D03} T. Ha-Duong, On retarded potential boundary integral equations and their discretisation, in Topics in Computational Wave Propagation: Direct and Inverse Problems, M. Ainsworth et al., eds., Springer Berlin Heidelberg, Berlin, Heidelberg, 2003, pp. 301--336.
\bibitem{EM77} B. Engquist, A. Majda, Absorbing boundary conditions for the numerical simulation of waves, Math. Comp. 31 (1977) 629--629.
\bibitem{E10} L.C. Evans, Partial Differential Equations (2nd edition), Graduate Studies in Mathematics, Vol. 19, American Mathematical Society, Providence, 2010.
\bibitem{FP96} D.A. French, T.E. Peterson, A continuous space-time finite element method for the wave equation, Math. Comp. 65 (1996) 491-506.
\bibitem{GSS06} M.J. Grote, A. Schneebeli, D. Sch\"otzau, Dicontinuous Galerkin finite element method for the wave equation, SIAM J. Numer. Anal. 44 (2006) 2408-2431.
\bibitem{HQSS17} M. Hassell, T. Qiu, T. S\'anchez-Vizuet, F.-J. Sayas, A new and improved analysis of the time domain boundary integral operators for acoustics, J. Integral Equations Applications 29(1) (2017) 107-136.
\bibitem{J64} D.S. Jones, The Theory of Electromagnetism, Pergamon Press Oxford, 1964.
\bibitem{LH21} I. Labarca, R. Hiptmair, Acoustic scattering problems with convolution quadrature and the method of fundamental solutions, Commun. Comput. Phys. 30 (2021) 985-1008.
\bibitem{LS09} A.R. Laliena, F.-J. Sayas, Theoretical aspects of the application of convolution quadrature to scattering of acoustic waves, Numer. Math. 112 (2009) 637-678.
\bibitem{LLC97} J.-F. Lee, R. Lee, A. Cangellaris, Time-domain finite-element methods, IEEE Trans. Antenn. Propag. 45 (1997) 430--442.
\bibitem{LZZ20} Y. Li, W. Zheng, X. Zhu, A CIP-FEM for high-frequency scattering problem with the truncated DtN boundary condition, CSIAM Trans. Appl. Math. 1(3) (2020) 530-560.
\bibitem{LB15} S. Lintner, O.P. Bruno, A generalized Calder\'on
  formula for open-arc diffraction problems: theoretical
  considerations, Proceedings of the Royal Society of Edinburgh
  Section A: Mathematics 145(2) (2015) 331-364.
\bibitem{L09} Y. Liu, Fast Multipole Boundary Element Method, Cambridge University Press, New York, 2009.
\bibitem{LSZ21} \textcolor{r1}{R. L\"oscher, O. Steinbach, M. Zank, Numerical results for an unconditionally stable space-time finite element method for the wave equation, arXiv:2103.04324.}
\bibitem{L94} C. Lubich, On the multistep time discretization of linear initial-boundary value problems and their boundary integral equations, Numer. Math. 67 (1994) 365-389.
\bibitem{M65} R. C. MacCamy, Low frequency acoustic oscillations, Q. Appl. Math. 23 (1965) 247-255.
\bibitem{MMRZ00} E. Mecocci, L. Misici, M. C. Recchioni, F. Zirilli, A new formalism for time-dependent wave scattering from a bounded obstacle, J. Acoust. Soc. Am. 107 (2000) 1825-1840.
\bibitem{Ne12} \textcolor{r1}{J. Ne\v{c}as, Direct Methods in the Theory of Elliptic Equations, Springer, 2012.}
\bibitem{PS88} \textcolor{r1}{T.v. Petersdorff, E.P. Stephan, A direct boundary element method for interface crack problems, In: S.N. Atluri, G. Yagawa (eds), Computational Mechanics '88, 1988, pp. 329-333, Springer.}
\bibitem{S16} F.-J. Sayas, Retarded Potentials and Time Domain Boundary Integral Equations, Springer International Publishing, 2016.
\bibitem{SU22} \textcolor{r1}{O. Steinbach, C. Urz\'ua-Torres, A new approach to space-time boundary integral equations for the wave equation, SIAM J. Math. Anal. 54(2) (2022) 1370-1392.}
\bibitem{SUZ21} \textcolor{r1}{O. Steinbach, C. Urz\'ua-Torres, M. Zank, Towards coercive boundary element methods for the wave equation, arXiv:2106.01646.}
\bibitem{SW84} E.P. Stephan, W.L. Wendland, An augmented Galerkin procedure for the boundary integral method applied to two-dimensional screen and crack problems, Appl. Anal. 3 (1984) 183-219.
\bibitem{S41} J.A. Stratton, Electromagnetic theory, McGraw-Hill, New York, 1941.
\bibitem{T00} A. Taflove, Computational electrodynamics: the finite-difference time-domain method, Artech House, Boston, 2000.
\bibitem{W86} P. Werner, Low frequency asymptotics for the reduced wave equation in two-dimensional exterior spaces, Math. Meth. Appl. Sci. 8 (1986): 134-156.
\bibitem{XCS13} Y. Xing, C.-S. Chou, C.-W. Shu, Energy conserving local discontinuous Galerkin methods for wave propagation problems, Inverse Probl. Imag. 7 (2013) 967-986.
\bibitem{Yilmaz:04} A.~Yilmaz, J.-M. Jin, E.~Michielssen, Time domain
adaptive integral method for surface integral equations, IEEE Trans. Antenn. Propag., 52 (2004), 2692-2708.
\bibitem{YBL20} X. Yuan, G. Bao, P. Li, An adaptive finite element DtN method for the open cavity scattering problems, CSIAM Trans. Appl. Math. 12 (2020) 316-345.
\bibitem{YMC22} Y. Yue, F. Ma, B. Chen, Time domain linear sampling method for inverse scattering problems with cracks, E. Asian J. Appl. Math. 12(1) (2022) 96-110.
\end{thebibliography}
\end{document}